\documentclass[11pt]{amsart}
\usepackage{amscd,amssymb}
\usepackage{amsfonts}
\usepackage{amsmath}
\usepackage{mathrsfs}
\usepackage{graphicx} 
\usepackage{color}
\usepackage{hyperref} 

\newtheorem{theorem}{Theorem}[section]
\newtheorem{corollary}[theorem]{Corollary}
 
\newtheorem{lemma}[theorem]{Lemma} 
 
\newtheorem{definition}[theorem]{Definition} 

\newtheorem{example}[theorem]{Example}

\newcommand\qbin[2]{\left[  \begin{array}{c}  #1 \\ #2  \end{array}  \right]}

\title[]{Congruence for lattice path models with filter restrictions and long steps}

\author{Dmitry Solovyev}
\address{D.S.: Euler International Mathematical Institute \\ Yau Mathematical Sciences Center, Tsinghua University, Beijing 100084, China \\ \& Department of Physics, Saint Petersburg State University,  Ulyanovkaya str.1, Saint Peterburg, Russia}
\email{dimsol42@gmail.com}
\begin{document}
	
\maketitle
	
\begin{abstract}
We derive a path counting formula for a two-dimensional lattice path model with filter restrictions in the presence of long steps, source and target points of which are situated near the filters. This solves the problem of finding an explicit formula for multiplicities of modules in tensor product decomposition of $T(1)^{\otimes N}$ for $U_q(sl_2)$ with divided powers, where $q$ is a root of unity. Combinatorial treatment of this problem calls for the definition of congruence of regions in lattice path models, properties of which are explored in this paper.
\end{abstract}
	
\tableofcontents
	
\section*{Introduction}
{Representation} theory of Kac--Moody algebras to this day serves as inspiration for numerous combinatorial problems, solutions to which give rise to interesting combinatorial structures. Examples of this can be met in (\cite{L95,M98,K90}) and many other well-known works. The problem of tensor power decomposition, in turn, can be considered from the combinatorial perspective as a problem of counting lattice paths in Weyl chambers (\cite{GM93,G02,TZ04,PR}). In this paper, we count paths on the Bratteli diagram \cite{B72}, reproducing the decomposition of tensor powers of the fundamental module of the quantum group $U_q(sl_2)$ with divided powers, where $q$ is a root of unity (\cite{QG,CP,A92,LPRS}), into indecomposable modules. Combinatorial treatment of this problem gives rise to some interesting structures on lattice path models, such as filter restrictions, first introduced in \cite{PS}, and long steps, which are introduced in the present paper.

In \cite{PS}, the considered lattice path model was motivated by the problem of finding explicit formulas for multiplicities of indecomposable modules in the decomposition of tensor power of fundamental module $T(1)$ of the small quantum group $u_q(sl_2)$ \cite{SQG}. We call this model the auxiliary lattice path model \cite{LPRS}. It consists of the left wall restriction at $x=0$ and filter restrictions located periodically at $x=nl-1$ for $n\in\mathbb{N}$. For $n=1$, the filter restriction is of type $1$, and the rest of the values of $n$ filter restrictions are of type $2$. Applying periodicity conditions $(M+2l,N)=(M,N)$, $M, N\geq l-1$ to the Bratteli diagram of this model allows one to obtain another lattice path model, recursion for weighted numbers of paths that coincide with recursion for multiplicities of indecomposable $u_q(sl_2)$-modules in the decomposition of $T(1)^{\otimes N}$. Counting weighted numbers of paths descending from $(0,0)$ to $(M,N)$ on this folded Bratteli diagram allows one to obtain desired formula for multiplicity, where $M$ stands for the highest weight of a module, the multiplicity of which is in question, and $N$ stands for the tensor power of $T(1)$. This has been performed in \cite{LPRS}.

We found that the auxiliary lattice path model can be modified in a different way, giving results for representation theory of $U_q(sl_2)$, the quantized universal enveloping algebra of $sl_2$ with divided powers, when $q$ is a root of unity \cite{S}. Instead of applying periodicity conditions to the auxiliary lattice path model, as in the case of $u_q(sl_2)$, for $U_q(sl_2)$ we consider all filters to be of the $1$st type and also allow additional steps from $x=nl-2$ to $x=(n-2)l-1$, where $n\geq 3$. Counting weighted numbers of paths descending from $(0,0)$ to $(M,N)$ on the Bratteli diagram of the lattice path model obtained by this modification gives a formula for the multiplicity of $T(M)$ in the decomposition of $T(1)^{\otimes N}$.

The main goal of this paper is to give a more in-depth combinatorial treatment of the auxiliary lattice path model in the presence of long steps and obtain explicit formulas for weighted numbers of paths, descending from $(0,0)$ to $(M,N)$. We explore combinatorial properties of long steps, as well as define boundaries and congruence of regions in lattice path models. Latter is found to be useful for deriving formulas for weighted numbers of paths. For any considered region, weighted numbers of paths at boundary points uniquely define such for the rest of the region by means of recursion. So, for congruent regions in different lattice path models, regions where, roughly speaking, recursion is similar, it is sufficient to prove identities only for boundary points of such regions.

This paper is organized as follows. In Section \ref{NOTATIONS}, we introduce the necessary notation. 
In Section \ref{ALPM}, we give background on the auxiliary lattice path model. In Section \ref{congandbound}, we introduce the notion of regions in lattice path models, boundary points and congruence of regions. In Section \ref{MALPM}, we explore combinatorial properties of long steps in periodically filtered lattice path models and consider the auxiliary lattice path model in the presence of long steps. We do so by means of boundary points and congruence of regions. In Section \ref{RelU}, we modify the auxiliary lattice path model and argue that the recursion for the weighted number of paths in such modified model coincides with the recursion for multiplicities of modules in tensor product decomposition of $T(1)^{\otimes N}$ for $U_q(sl_2)$ with divided powers, where $q$ is a root of unity. In Section \ref{PATHS}, we prove formulas for the weighted numbers of descending paths, relevant to this modified model. In Section \ref{CONCLUSIONS}, we conclude this paper with observations for possible future directions or research.
\subsection*{Acknowledgements} We are grateful to Olga Postnova, Nicolai Reshetikhin, Pavel Nikitin, Fedor Petrov for fruitful discussions. This research was supported by the grant "Agreement on the provision of a grant in the form of subsidies from the federal budget for implementation of state support for the creation and development of world-class scientific centers, including international world-class centers and scientific centers, carrying out research and development on the priorities of scientific technological development No.~075-15-2019-1619" dated 8th November 2019. The author expresses gratitude for the support by the Xing Hua Scholarship program of Tsinghua University.\par

\section{Notations}\label{NOTATIONS}
In this paper, we use the notation following \cite{Kratt}. For our purposes of counting multiplicities in tensor power decomposition of $U_q(sl_2)$-module $T(1)$, throughout this paper, we consider the lattice
$$\mathcal{L}=\{(n,m)| n+m=0\; \text{mod} \;2\}\subset \mathbb{Z}^2,$$
and the set of steps $\mathbb{S}=\mathbb{S}_L\cup \mathbb{S}_R$, where
$$\mathbb{S}_R=\{(x,y)\to(x+1,y+1)\},\; \mathbb{S}_L=\{(x,y)\to(x-1,y+1)\}.$$
{A}
\textit{{lattice path}
} $\mathcal{P}$  in $\mathcal{L}$  is a sequence $\mathcal{P}=(P_0,P_1,\dots, P_m)$ of points $P_i=(x_i,y_i)$ in $\mathcal{L}$ with starting point $P_0$ and the endpoint $P_m$. The pairs $P_0\to P_1, P_1\to P_2\dots P_{m-1}\to P_m$ are called steps of $\mathcal{P}$.

Given starting point $A$ and endpoint $B$, a set $\mathbb{S}$ of steps and a set of restrictions $\mathcal{C}$ \mbox{we write} 
$$L(A\to B;\mathbb{S}\;|\;\mathcal{C})$$
for the set of all lattice paths from $A$ to $B$ that have steps from $\mathbb{S}$ and obey the restrictions from $\mathcal{C}$. We denote the number of paths in this set as
$$|L(A\to B;\mathbb{S}\;|\;\mathcal{C})|.$$

The set of restrictions $\mathcal{C}$ in lattice path models considered throughout this paper mostly contain wall restrictions and filter restrictions. Left(right) wall restrictions forbid steps in the left(right) direction, reflecting descending paths and preventing them from crossing the 'wall'. Filter restrictions forbid steps in certain directions and provide other steps with non-uniform weights, so paths can cross the 'filter' in one direction, but cannot cross it in the opposite direction. A rigorous definition of these restrictions is given in subsequent sections.

To each step from $(x,y)$ to $(\tilde{x},\tilde{y})$  we assign the weight  function $\omega:\mathbb{S}\longrightarrow \mathbb{R}_{>0}$ and use notation  $(x,y)\xrightarrow[]{\omega}(\tilde{x},\tilde{y})$  to denote that  the step  from $(x,y)$ to $(\tilde{x},\tilde{y})$ has the weight $\omega$. By default, all unrestricted steps from $\mathbb{S}$ will have weight $1$  and is denoted by an arrow with no number at the top. The \textit{weight} of a path $\mathcal{P}$ is defined as the product
$$\omega(\mathcal{P})=\prod_{i=0}^{m-1}\omega(P_i\to P_{i+1}).$$

For the set $L(A\to B;\mathbb{S}\;|\;\mathcal{C})$ we define the \textit{weighted number of paths} as
$$Z(L(A\to B;\mathbb{S}\;|\;\mathcal{C}))=\sum_{\mathcal{P}}\omega(\mathcal{P}),$$
where the sum is taken over all paths $\mathcal{P}\in L(A\to B;\mathbb{S}\;|\;\mathcal{C})$.

\section{The Auxiliary Lattice Path Model}\label{ALPM}
In this section, we briefly revise notions and results obtained in \cite{PS}, relevant for future considerations. It is convenient for us to omit mentioning $\mathbb{S}$ in $L(A\to B;\mathbb{S}\;|\;\mathcal{C})$. All paths considered below involve steps from set $\mathbb{S}$ unless stated otherwise.

\subsection{Unrestricted Paths}
Let $L(A\to B)$ be the set of unrestricted paths from $A$ to $B$ on lattice $\mathcal{L}$ with the steps $\mathbb{S}$. An example of such a path is given in Figure \ref{up}.
\begin{figure}[h!]
	{\includegraphics[width=250pt]{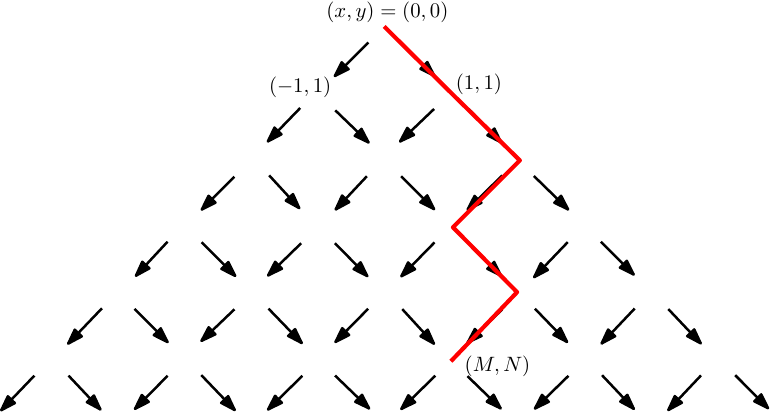}}
	\caption{Example of an unrestricted path in $L((0,0)\to (M,N))$ for lattice $\mathcal{L}$ and set of steps $\mathbb{S}$.}
	\label{up}
\end{figure}
\begin{lemma}
	For a set of unrestricted paths with steps $\mathbb{S}$ we have
	\begin{equation}
	|L((0,0)\to (M,N))|={N\choose\frac{N-M}{2}}.
	\end{equation}
\end{lemma}

\subsection{Wall Restrictions}
\begin{definition} For lattice paths that start at $(0,0)$ we will say  that  $\mathcal{W}^L_d$ with $d\leq 0$  is a \textit{left wall restriction} (relative to $x=0$) if at points $(d,y)$ paths are allowed to take steps of type $\mathbb{S}_R$ only
	\begin{equation*}
	\mathcal{W}^L_{d}=\{(d,y)\to (d+1,y+1)\}.
	\end{equation*}
\end{definition}  
\begin{lemma}
	The number of paths from $(0,0)$ to $(M,N)$ with the set of steps $\mathbb{S}$ 
	and one wall restriction $\mathcal{W}^L_{a}$ can be expressed via the number of unrestricted paths as
	\begin{equation}
	|L((0,0)\to (M,N)\;|\;\mathcal{W}^L_a)|={N\choose\frac{N-M}{2}}-{N\choose\frac{N-M}{2}+a-1},\;\;\text{ for}\;\; M\geq a,
	\end{equation}
\end{lemma}
We considered the left walls located at $x=0$. An example of possible steps for paths descending from $(0,0)$ in the presence of this restriction is given in Figure \ref{wr}.
\begin{figure}[h!]
	{\includegraphics[width=150pt]{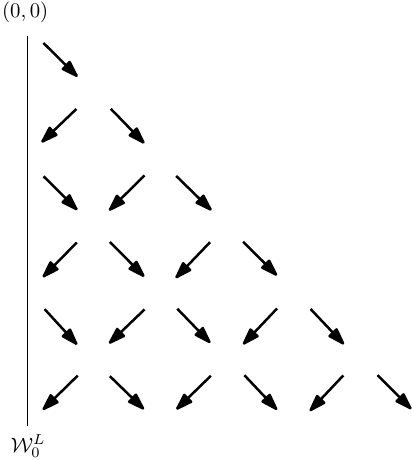}}
	\caption{Arrangement of steps for points of $\mathcal{L}$ in presence of restriction $\mathcal{W}^L_0$.}
	\label{wr}
\end{figure}

\subsection{Filter Restrictions}
\begin{definition}
	For $n\in\mathbb{N}$, we say that there is a filter $\mathcal{F}_d^{n}$ of type $n$, located at $x=d$ if at $x=d-1,\;d,\;d+1$ only the following steps are allowed: 
	\begin{eqnarray*}
		\mathcal{F}_{d}^{n}&=&\{(d-1,y-1)\xrightarrow[]{n}(d,y),\; (d-1,y-1)\to(d-2,y),\\
		& &(d,y)\to(d+1,y+1),\; (d+1,y+1)\to(d+2,y+2),\; (d+1,y+1)\xrightarrow[]{2}(d,y+2)\}.
	\end{eqnarray*}
	The index above the arrow is the weight of the step.
\end{definition}
Note that by default, an arrow with no number at the top means that the corresponding step has a weight of $1$. An example of possible steps for descending paths in the presence of this restriction is given in Figure \ref{F1}. We highlighted steps of weight $2$ with red instead of an arrow with a superscript $2$ for future convenience, as those are the most common for the auxiliary lattice path model and its modifications. We were mostly involved with filters of type $1$, so superscripts $n$ were avoided, leaving Bratteli diagrams with black and red arrows, with weights $1$ and $2$ correspondingly.
\begin{figure}[h!]
	{\includegraphics[width=210pt]{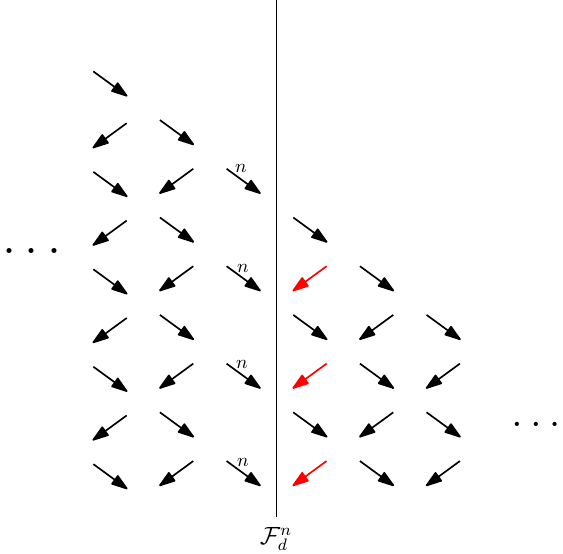}}
	\caption{Filter $\mathcal{F}_d^{n}$. Red arrows correspond to steps $(d+1,y+1)\xrightarrow[]{2}(d,y+2)$  that has a \mbox{weight $2$}. Black arrows with superscript $n$ correspond to steps $(d,y)\xrightarrow[]{n}(d+1,y+1)$. Other steps have \mbox{weight $1$}.}
	\label{F1}
\end{figure}
\begin{lemma}\label{counting_lemma}
	The number of lattice paths from $(0,0)$  to $(M,N)$  with steps from $\mathbb{S}$ and filter restriction $\mathcal{F}_d^{n}$ with $x=d>0$ and $n\in\mathbb{N}$ is
	\begin{equation}
	Z(L_N((0,0)\to (M,N)\;|\;\mathcal{F}_d^{n}))={N\choose\frac{N-M}{2}}- {N\choose\frac{N-M}{2}+d},\;\text{for}\; M<d, 
	\end{equation}
	\begin{equation}
	Z(L_N((0,0)\to (M,N)\;|\;\mathcal{F}_d^{n}))=n {N\choose\frac{N-M}{2}} ,\;\text{for}\; M>d.
	\end{equation}
\end{lemma}
\begin{proof}
	The proof is the same as for Lemma 4.8 and Lemma 4.9 in \cite{PS}.
\end{proof}

\subsection{Counting Paths in the Auxiliary Lattice Path Model}
Consider the lattice path model for the set of paths on $\mathcal{L}$ descending from $(0,0)$ to $(M,N)$ with steps $\mathbb{S}$ in the presence of restrictions $\mathcal{W}^L_0$, $\mathcal{F}^1_{l-1}$, $\mathcal{F}^2_{nl-1}$, $n\in\mathbb{N}$, $n\geq 2$. Such set is denoted as
\begin{equation*}
L_N((0,0)\to (M,N);\mathbb{S}\;|\;\mathcal{W}^L_0, \mathcal{F}^1_{l-1}, \{\mathcal{F}^2_{nl-1}\}_{n=2}^{\infty})
\end{equation*}
and such a model is called the auxiliary lattice path model. The main theorem of \cite{PS} gives an explicit formula for weighted numbers of paths in the auxiliary lattice path model. Then, in \cite{LPRS}, periodicity conditions $(M+2l,N)=(M,N)$, $M, N\geq l-1$ were applied, resulting in a folded Brattelli diagram. For such a diagram, recursion on the weighted numbers of paths coincides with recursion on multiplicities of indecomposable $u_q(sl_2)$-modules in tensor product decomposition of $T(1)^{\otimes N}$. Note that due to properties of the category $\mathbf{Rep}(u_q(sl_2))$, we mostly considered odd values of $l$; however, the results remain to be true for even values of $l$ as well.

Before coming to modifications of the auxiliary lattice path model relevant to the representation theory of $U_q(sl_2)$ at the roots of unity, we need to slightly tweak it. We are interested in paths descending from $(0,0)$ to $(M,N)$ with steps $\mathbb{S}$ in the presence of restrictions $\mathcal{W}^L_0$, $\mathcal{F}^1_{nl-1}$, $n\in\mathbb{N}$, instead of filters of type 2. Such lattice path model is depicted in Figure \ref{stalpm}.
\begin{figure}[h!]
	{\includegraphics[width=340pt]{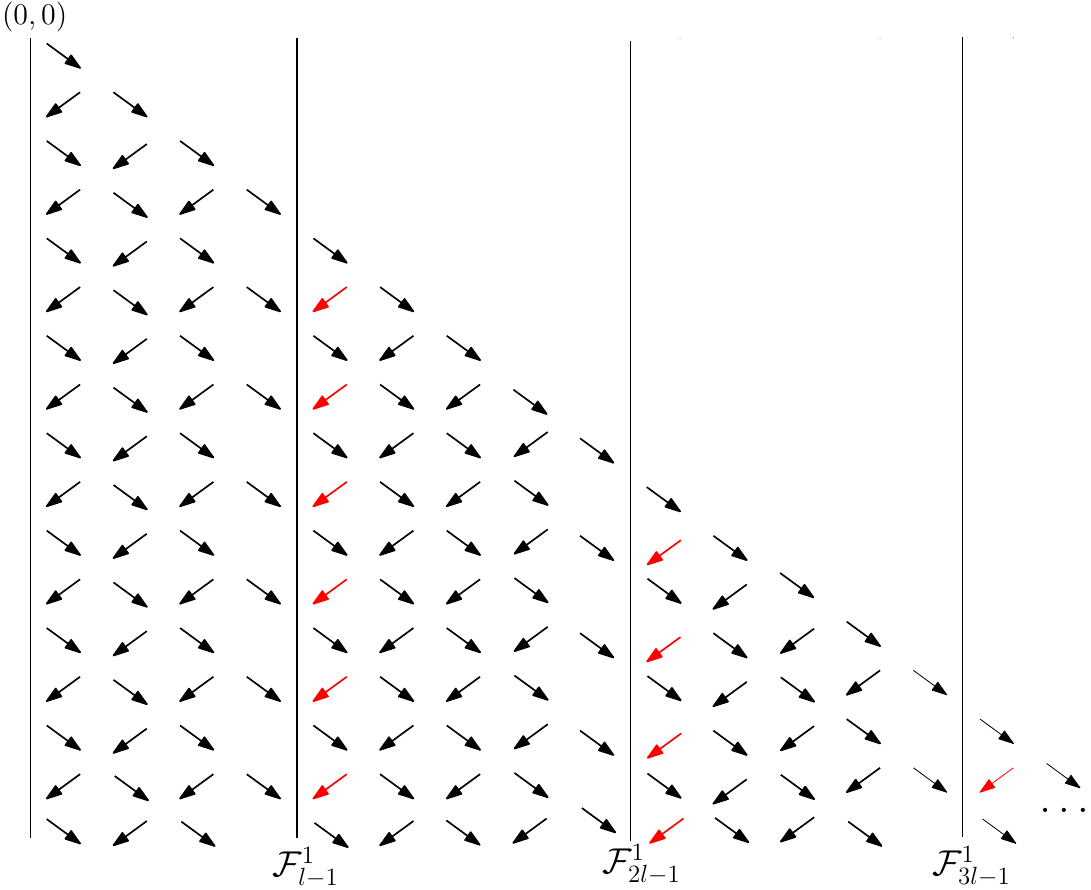}}
	\caption{{Arrangement} 
		of steps for points of $\mathcal{L}$ in the considered, slightly tweaked version of the auxiliary lattice path model. Here, we depict the case, where $l=5$.}
	\label{stalpm}
\end{figure}
\begin{definition}
	We denote by multiplicity function in the $j$-th strip $M^j_{(M,N)}$ the weighted number of paths in set 
	\begin{equation*}
	L_N((0,0)\to (M,N);\mathbb{S}\;|\;\mathcal{W}^L_0,\{\mathcal{F}_{nl-1}^{1}\}, n\in \mathbb{N})
	\end{equation*}
	with the endpoint $(M,N)$ that lies within $(j-1)l-1 \leq M \leq jl-2$
	\begin{equation}
	M^j_{(M,N)} = Z(L_N((0,0)\to (M,N);\mathbb{S}\;|\;\mathcal{W}^L_0,\{\mathcal{F}_{nl-1}^{1}\}, n\in \mathbb{N})),
	\end{equation}
	where $M\geq 0$ and $j=\Big[\frac{M+1}{l}+1\Big]$. 
\end{definition}
Now consider the version of the main theorem in \cite{PS} corresponding to this model.
\begin{theorem}[\cite{PS}]\label{mainps}
	The multiplicity function in the $j$-th strip  is given by
	\begin{eqnarray*}
		M^{j}_{(M,N)} =\sum_{k=0}^{\big[\frac{N-(j-1)l+1}{4l}\big]}P_j(k)F^{(N)}_{M+4kl}+\sum_{k=0}^{\big[\frac{N-jl}{4l}\big]}P_j(k)F^{(N)}_{M-4kl-2jl}-\\
		-\sum_{k=0}^{\big[\frac{N-(j+1)l+1}{4l}\big]}Q_j(k)F^{(N)}_{M+2l+4kl}-\sum_{k=0}^{\big[\frac{N-jl-2l}{4l}\big]}Q_j(k)F^{(N)}_{M-4kl-2(j+1)l},
	\end{eqnarray*}
	where
	\begin{equation}
	P_j(k)=\sum_{i=0}^{\big[\frac{j}{2}\big]}\binom{j-2}{2i}\binom{k-i+j-2}{j-2},\;\;\;\;\;
	Q_j(k)=\sum_{i=0}^{\big[\frac{j}{2}\big]}\binom{j-2}{2i+1}\binom{k-i+j-2}{j-2},
	\end{equation}
	\begin{equation*}
	F_M^{(N)}= {N\choose\frac{N-M}{2}}-{N\choose{\frac{N-M}{2}-1}}.
	\end{equation*}
\end{theorem}
\begin{proof}
	The proof is the same as the proof of the main theorem in \cite{PS}, except that instead of Lemma 4.9 in \cite{PS}, for the slightly tweaked model one should use Lemma \ref{counting_lemma}.
\end{proof}
From now on, when mentioning the auxiliary lattice path model, we mean its slightly tweaked version. This model will be further modified in subsequent sections. Instead of applying periodicity conditions, as for $u_q(sl_2)$, we enhance this model with long steps, source and target points which are located near filters. As a result, recursion for the weighted numbers of paths on the resultant Bratteli diagram recreates recursion for multiplicities of indecomposable $U_q(sl_2)$-modules in the decomposition of $T(1)^{\otimes N}$.

\section{Boundary Points and Congruent Regions}\label{congandbound}
In this section, we consider notions, which are convenient for counting paths in the auxiliary lattice path model in the presence of long steps. We will see, that multiplicities on the boundary of a region uniquely define multiplicities in the rest of the region. For proving identities between multiplicities in two congruent regions, it is sufficient to prove such identities for their boundary points.
\begin{definition}
	Consider the lattice path model, defined by a set of steps $\mathbb{S}$ and a set of restrictions $\mathcal{C}$ on lattice $\mathcal{L}$. Subset $\mathcal{L}_0\subset\mathcal{L}$ with steps $\mathbb{S}$ and restrictions $\mathcal{C}$ is called a region of the lattice path model under consideration.
\end{definition}
Intuitively, region $\mathcal{L}_0\subset\mathcal{L}$ is a restriction of the lattice path model defined by $\mathbb{S}$, $\mathcal{C}$ on lattice $\mathcal{L}$ to the subset $\mathcal{L}_0$. The word 'restriction' is overused, so we consider regions of lattice path models instead.
\begin{definition}\label{bounddef}
	Consider $\mathcal{L}_0\subset\mathcal{L}$ a region of the lattice path model defined by steps $\mathbb{S}$ and restrictions $\mathcal{C}$. Point $B\in\mathcal{L}_0$ is called a boundary point of $\mathcal{L}_0$ if there exists $B^\prime\in\mathcal{L}$, $B^\prime\notin\mathcal{L}_0$ such that step $B^\prime \to B$ is allowed in $\mathcal{L}$ by a set of steps $\mathbb{S}$ and restrictions $\mathcal{C}$. The union of all such points is a boundary of $\mathcal{L}_0$ and is denoted by $\partial\mathcal{L}_0$.
\end{definition}
The definition \ref{bounddef} introduces a notion, reminiscent of the outer boundary in graph theory. Note that boundary points are defined with respect to some lattice path models under consideration. For brevity, we assume that this lattice path model is known from the context, and mentioning it will be mostly omitted.
\begin{example}
	For a strip in the auxiliary lattice path model, its boundary is in the left filter. It is depicted in Figure \ref{bstrip}.
	\begin{figure}[h!]
		{\includegraphics[width=280pt]{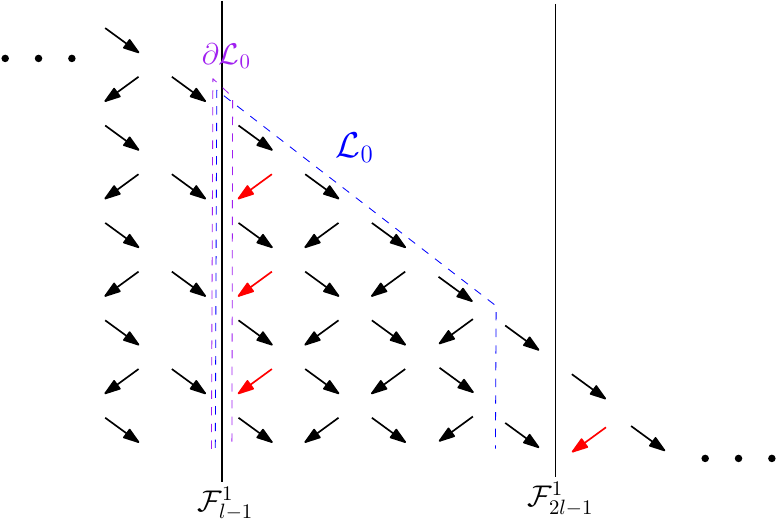}}
		\caption{Region $\mathcal{L}_0$, highlighted with blue dashed lines, is a $2$nd strip for $l=5$. Its boundary $\partial\mathcal{L}_0$ is a set of points in the left filter restriction $\mathcal{F}^1_{l-1}$, which is highlighted with purple dashed lines.} 
		\label{bstrip}
	\end{figure}		
\end{example}
\begin{example}
	Consider region $\mathcal{L}_0$ of the unrestricted lattice path model, as depicted in Figure \ref{bpascal} and highlighted with blue dashed lines. Its boundary is a set of points highlighted with purple \mbox{dashed lines}. 
	\begin{figure}[h!]
		{\includegraphics[width=280pt]{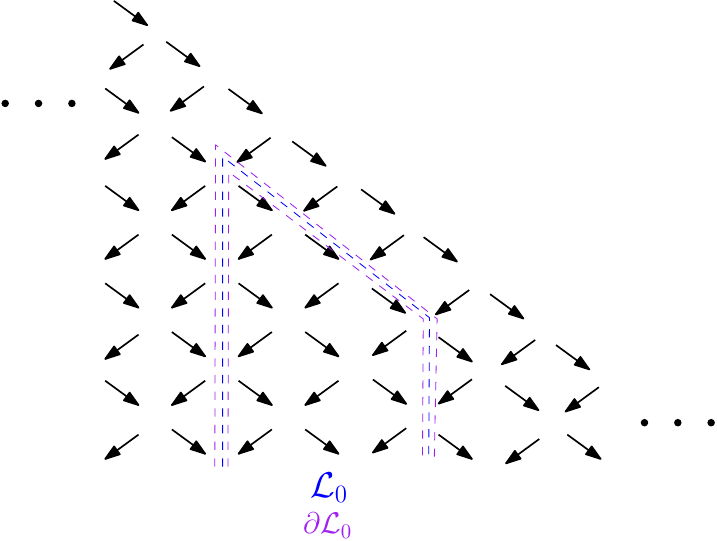}}
		\caption{Region $\mathcal{L}_0$ is highlighted with blue dashed lines. Its boundary $\partial\mathcal{L}_0$ is a set of points highlighted with purple dashed lines.}
		\label{bpascal}
	\end{figure}
\end{example}
\begin{lemma}\label{BD} 
	Consider region $\mathcal{L}_0$ of a lattice path model defined by  $\mathbb{S}$, $\mathcal{C}$ on lattice $\mathcal{L}$. Weighted numbers of paths $Z(L_N((0,0)\to (M,N);\ldots)$ for $(M,N)\in\mathcal{L}_0$ are uniquely defined by weighted numbers of paths for its boundary points $\partial\mathcal{L}_0$.
\end{lemma}
\begin{proof}
	Suppose weighted numbers of paths for $\partial\mathcal{L}_0$ are known. Suppose that there exists some point $A\in\mathcal{L}_0$, such that its weighted number of paths cannot be expressed in terms of weighted numbers of paths for points in $\partial\mathcal{L}_0$.\par 
	The first case is that recursion for a weighted number of paths for $A$ involves some point $A^\prime\in\mathcal{L}_0$, a weighted number of paths for which cannot be expressed in terms of such for points in $\partial\mathcal{L}_0$. In this case, we need to consider $A^\prime$ and recursion on the weighted number of paths for such a point instead of $A$.\par
	The second case is that recursion for a weighted number of paths for $A$ involves a weighted number of paths for some point $A^\prime\notin\mathcal{L}_0$. Then, $A\in\partial\mathcal{L}_0$ by definition of a boundary point and weighted number of paths for such point is known by the initial supposition of the lemma.\par
	Note that due to the fact that we consider descending paths, $M$ and $N$, to be finite, the first case can be iterated finitely many times at most. 
\end{proof} 

\begin{definition}\label{defcong}
	Consider two lattice path models with steps $\mathbb{S}_1$, $\mathbb{S}_2$ and restrictions $\mathcal{C}_1$, $\mathcal{C}_2$ defined on lattice $\mathcal{L}$. Subset $\mathcal{L}_1\subset\mathcal{L}$ is a region in the lattice path model defined by $\mathbb{S}_1,\mathcal{C}_1$. Subset $\mathcal{L}_2\subset\mathcal{L}$ is a region in the lattice path model defined by $\mathbb{S}_2,\mathcal{C}_2$. Regions $\mathcal{L}_1$ and $\mathcal{L}_2$ are congruent if there exists a translation $T$ in $\mathcal{L}$ such that
	\begin{itemize}
		\item $T\mathcal{L}_1=\mathcal{L}_2$ as sets of points in $\mathcal{L}$
		\item Translation $T$ induces a bijection between steps in $\mathcal{L}_1$ and $\mathcal{L}_2$, meaning that there is a one-to-one correspondence between steps with source and target points related by $T$, with preservation of weights.
	\end{itemize}
\end{definition}
The second condition can be written down explicitly. Firstly, for each $(M,N)\in\mathcal{L}_1$ and each step $(M,N)\xrightarrow{w}(P,Q)$ in $\mathbb{S}_1$ obeying $\mathcal{C}_1$ such that $(P,Q)\in\mathcal{L}_1$, there is a step $(M^\prime,N^\prime)\xrightarrow{w}(P^\prime,Q^\prime)$ in $\mathbb{S}_2$ obeying $\mathcal{C}_2$, where $T(M,N)=(M^\prime,N^\prime)$, $T(P,Q)=(P^\prime,Q^\prime)$. Secondly, for each $(M^\prime,N^\prime)\in\mathcal{L}_2$ and each step $(M^\prime,N^\prime)\xrightarrow{w}(P^\prime,Q^\prime)$ in $\mathbb{S}_2$ obeying $\mathcal{C}_2$ such that $(P^\prime,Q^\prime)\in\mathcal{L}_2$, there is a step $(M,N)\xrightarrow{w}(P,Q)$ in $\mathbb{S}_1$ obeying $\mathcal{C}_1$, where $T^{-1}(M^\prime,N^\prime)=(M,N)$, $T^{-1}(P^\prime,Q^\prime)=(M,N)$. To put it simply, if we forget about lattice path models outside $\mathcal{L}_1$ and $\mathcal{L}_2$, these two regions will be indistinguishable. Due to translations in $\mathcal{L}$ being invertible, it is easy to see that congruence defines an equivalence relation.

Now we must prove the main theorem of this subsection.
\begin{theorem}\label{BD_cor}
	Consider two lattice path models with steps $\mathbb{S}_1$, $\mathbb{S}_2$ and restrictions $\mathcal{C}_1$, $\mathcal{C}_2$ defined on lattice $\mathcal{L}$. Region $\mathcal{L}_1$ of the lattice path model defined by $\mathbb{S}_1$, $\mathcal{C}_1$ is congruent to region $\mathcal{L}_2$ of the lattice path model defined by $\mathbb{S}_2$, $\mathcal{C}_2$, where $T\mathcal{L}_1=\mathcal{L}_2$. If equality
	\begin{equation}\label{eqBD0}
	Z(L_N((0,0)\to (M,N);\mathbb{S}_1\;|\;\mathcal{C}_1))=Z(L_N((0,0)\to T(M,N);\mathbb{S}_2\;|\;\mathcal{C}_2))
	\end{equation}
	holds for all $(M,N)\in\partial\mathcal{L}_1\cup T^{-1}(\partial\mathcal{L}_2)$, then it holds for all $(M,N)\in\mathcal{L}_1$.
\end{theorem}
Note, that if $(M,N)\in\partial\mathcal{L}_1$ it does not necessarily follow that $T(M,N)\in\partial\mathcal{L}_2$, due to $\mathcal{C}_1$ and $\mathcal{C}_2$ being different. So, it is natural to ask Formula (\ref{eqBD0}) to hold for $\partial\mathcal{L}_1\cup T^{-1}(\partial\mathcal{L}_2)$.
\begin{proof}
	We need to prove that Formula (\ref{eqBD0}) is true for $(M,N)\in\mathcal{L}_1$. The l.h.s. can be uniquely expressed in terms of its values at $\partial\mathcal{L}_1\cup T^{-1}(\partial\mathcal{L}_2)$, following procedure in Lemma \ref{BD}. Due to the congruence between $\mathcal{L}_1$ and $\mathcal{L}_2$, recursion for the r.h.s. of (\ref{eqBD0}) coincides with the one for the l.h.s., so we can obtain the same expression on the r.h.s., but with values of weighted numbers of paths for $T(\partial\mathcal{L}_1\cup T^{-1}(\partial\mathcal{L}_2))=T(\partial\mathcal{L}_1)\cup\partial\mathcal{L}_2$ instead of $\partial\mathcal{L}_1\cup T^{-1}(\partial\mathcal{L}_2)$. We can compare the l.h.s. and the r.h.s. term by term, for points related by translation $T$. All of such terms have the same values due to the initial supposition of the theorem.
\end{proof}
\begin{corollary}\label{BD_cor_cor}
	Consider lattice path models with steps $\mathbb{S}_1$, $\mathbb{S}_2$, $\mathbb{S}_3$ and restrictions $\mathcal{C}_1$, $\mathcal{C}_2$, $\mathcal{C}_3$ defined on lattice $\mathcal{L}$. Region $\mathcal{L}_1$ is congruent to $\mathcal{L}_2$ and $\mathcal{L}_3$, where $T_1(M_1,N_1)=(M_2,N_2)$, $T_2(M_1,N_1)=(M_3,N_3)$ for $(M_1,N_1)\in\mathcal{L}_1$. If equality
	\begin{eqnarray}\label{eqBD_cor_cor}
	& Z(L_N((0,0)\to (M,N);\mathbb{S}_1\;|\;\mathcal{C}_1)) = \nonumber \\
	& =Z(L_N((0,0)\to T_1(M,N);\mathbb{S}_2\;|\;\mathcal{C}_2))+Z(L_N((0,0)\to T_2(M,N);\mathbb{S}_3\;|\;\mathcal{C}_3))
	\end{eqnarray}
	holds for all $(M,N)\in\partial\mathcal{L}_1\cup T_1^{-1}(\partial\mathcal{L}_2)\cup T_2^{-1}(\partial\mathcal{L}_3)$, then it holds for all $(M,N)\in \mathcal{L}_1$.
\end{corollary}
\begin{proof}
	Due to linearity of the r.h.s. of Formula (\ref{eqBD_cor_cor}), the proof repeats the one of \mbox{Theorem \ref{BD_cor}}.
\end{proof}
The moral of this section is that for two congruent regions, weighted numbers of paths are defined by values of such at the boundary of the considered regions. For proving identities, it is sufficient to establish equality for weighted numbers of paths at boundary points, while equality for the rest of the region will follow due to the congruence.

\section{The Auxiliary Lattice Path Model in the Presence of Long Steps}\label{MALPM}
\subsection*{Long Steps in Lattice Path Models with Filter Restrictions}
Long step is a step $(x,y)\xrightarrow[]{w}(x^\prime,y+1)$ in $\mathcal{L}$ such that $|x-x^\prime|>1$. We denote the sequence of long steps as
\begin{equation*}
\mathbb{S}[M_1,M_2]=\{(M_1,M_1+2m)\to(M_2,M_1+1+2m)\}_{m=0}^{\infty},
\end{equation*}
where $x=M_1$ is the source point for the sequence and $x=M_2$ is the target point, $|M_1-M_2|>1$. For the purposes of this paper, we are mainly interested in sequences
\begin{equation*}
\mathbb{S}(k)\equiv\mathbb{S}[l(k+2)-2,lk-1]=\{(l(k+2)-2,lk-2+2m)\to(lk-1,lk-2+1+2m)\}_{m=0}^{\infty}, 
\end{equation*}
where $k\in\mathbb{N}$ and $\mathcal{C}$ consists of $\mathcal{F}^1_{lk-1}$ and $\mathcal{F}^1_{l(k+2)-1}$. We need such sequences of long steps for modification of the auxiliary lattice path model, relevant to the representation theory of $U_q(sl_2)$ at roots of unity.
\begin{lemma} \label{lstep_property}
	Fix $k\in\mathbb{N}$. Let
	\begin{equation*}
	Z_{(M,N)}\equiv Z(L_N((0,0)\to(M,N));\mathbb{S}\;|\;\mathcal{F}^1_{lk-1}, \mathcal{F}^1_{l(k+2)-1})
	\end{equation*}
	be the weighted number of lattice paths from $(0,0)$ to $(M,N)$ with filter restrictions $\mathcal{F}^1_{lk-1}$, $\mathcal{F}^1_{l(k+2)-1}$ and set of unrestricted elementary steps $\mathbb{S}$. Let
	\begin{equation*}
	Z^\prime_{(M,N)}\equiv Z(L_N((0,0)\to(M,N));\mathbb{S}\cup\mathbb{S}(k)\;|\;\mathcal{F}^1_{lk-1}, \mathcal{F}^1_{l(k+2)-1})
	\end{equation*}
	be the weighted number of lattice paths from $(0,0)$ to $(M,N)$ with the same restrictions, with steps $\mathbb{S}\cup\mathbb{S}(k)$.
	Then for $lk-1\leq M\leq l(k+2)-2$ we have
	\begin{equation}\label{eqq1}
	Z^\prime_{(M,N)}=Z_{(M,N)},\quad \text{if $N\leq M+ 2l-2$},
	\end{equation}
	\begin{equation}\label{eqq2}
	Z^\prime_{(M,N)}=Z_{(M,N)}+Z_{(M+2l,N)},\quad \text{if  $M+ 2l\leq N \leq l(k+4)-2$}.
	\end{equation}
\end{lemma}
\begin{proof}
	In Figure \ref{l88}, we depict the setting of the Lemma \ref{lstep_property}. Long steps do not impact region I, so Formula (\ref{eqq1}) is true.
	\begin{figure}[h!]
		{\includegraphics[width=300pt]{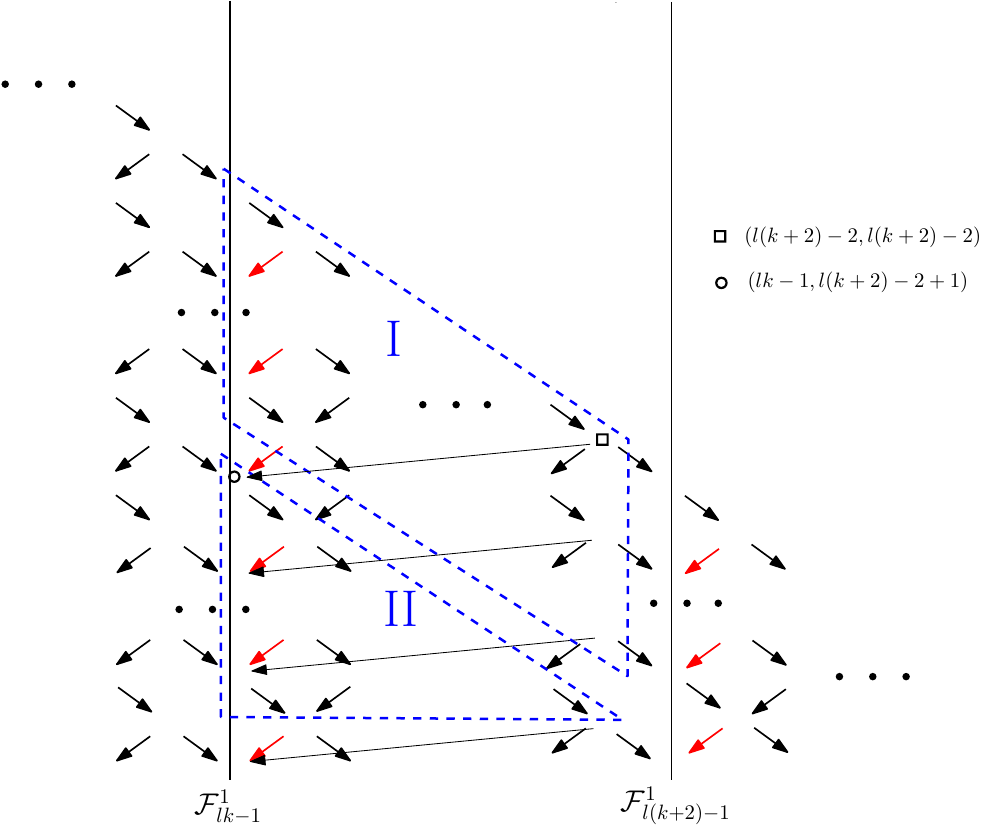}}
		\caption{By square and circle we denote points, where long steps first appear. Regions I and II highlighted with blue dashed lines correspond to cases $N\leq M+ 2l-2$, as in (\ref{eqq1}), and $M+ 2l\leq N \leq l(k+4)-2$, as in (\ref{eqq2}).}
		\label{l88}
	\end{figure}

	Consider Formula (\ref{eqq2}). The weighted number of paths in the l.h.s. involves points from region II. Its boundary contains points of the left cathetus of region II, of the form $(lk-1,N)$ for $l(k+2)-1\leq N\leq l(k+4)-2$, and points of the hypotenuse of the region II, of the form $(lk-1+j,l(k+2)-1+j)$ for $j=1,\ldots,2l-1$. Denote this set by $\partial\mathcal{L}_{II}$. The r.h.s. of (\ref{eqq2}) has two terms. The first involves region II, the boundary of which we have already considered. The second term involves points of the region congruent to region II, as they are related by translation $T(M,N)=(M+2l,N)$, satisfying Definition \ref{defcong}. Its boundary consists of the image of the left cathetus of region II under translation $T$. Denote this set by $\partial\mathcal{L}^\prime_{II}$. By Corollary \ref{BD_cor_cor}, it is sufficient to prove Formula (\ref{eqq2}) for $\partial\mathcal{L}_{II}\cup T^{-1}(\partial\mathcal{L}^\prime_{II})=\partial\mathcal{L}_{II}$.

	We proceed by induction over $n$, where $N=l(k+2)-1+2n$. For $n=0$ from recursion we have
	\begin{equation}
	Z^\prime_{(lk-1,l(k+2)-1)}=Z_{(lk-2,l(k+2)-2)}+2Z_{(lk,l(k+2)-2)}+Z_{(l(k+2)-2,l(k+2)-2)},
	\end{equation}
	which, taking into account that
	\begin{equation*}
	Z_{(lk-2,l(k+2)-2)}+2Z_{(lk,l(k+2)-2)}=Z_{(lk-1,l(k+2)-1)},
	\end{equation*}
	\begin{equation*}
	Z_{(l(k+2)-2,l(k+2)-2)}=Z_{(l(k+2)-1,l(k+2)-1)},
	\end{equation*}
	gives us
	\begin{equation}
	Z^\prime_{(lk-1,l(k+2)-1)}=Z_{(lk-1,l(k+2)-1)}+Z_{(l(k+2)-1,l(k+2)-1)}.
	\end{equation}
	{We} obtained the base of induction.

	In a similar manner, it also follows, that Formula (\ref{eqq2}) is true for boundary points of the hypotenuse of region II. In order to show this, one must consider recursion explicitly and use the fact that
	\begin{equation}
	Z_{(j,j)}=Z_{(k,k)},\quad\text{for all $j,k>0$}.
	\end{equation}
	{Now} it is sufficient to prove Formula (\ref{eqq2}) for boundary points, situated in the left cathetus of region II.
	
	Suppose
	\begin{equation}
	Z^\prime_{(lk-1,l(k+2)-1+2n)}=Z_{(lk-1,l(k+2)-1+2n)}+Z_{(l(k+2)-1,l(k+2)-1+2n)}
	\end{equation}
	is true. For the sake of brevity, we rewrite this expression as
	\begin{equation}
	Z^\prime_{(p,q+2n)}=Z_{(p,q+2n)}+Z_{(q,q+2n)},
	\end{equation}
	where $p=lk-1$, $q=l(k+2)-1$, $q=p+2l$. 
	By Theorem \ref{BD_cor}, it follows that Formula (\ref{eqq2}) is true for the region, corresponding to boundary points, covered by the inductive supposition. In particular, this region includes points $(p+j,q+2n+j)$ for $j=0,\ldots,2l-1$. Need to prove that
	\begin{equation}
	Z^\prime_{(p,q+2(n+1))}=Z_{(p,q+2(n+1))}+Z_{(q,q+2(n+1))}
	\end{equation}
	{Taking} into account, that
	\begin{equation*}
	Z^\prime_{(p,q+2n+2)}=Z_{(p-1,q+2n+1)}+2Z^\prime_{(p+1,q+2n+1)}+Z_{(q-1,q+2n+1)},
	\end{equation*}
	\begin{equation*}
	Z_{(p,q+2n+2)}=Z_{(p-1,q+2n+1)}+2Z_{(p+1,q+2n+1)},
	\end{equation*}
	\begin{equation*}
	Z_{(q,q+2n+2)}=Z_{(q-1,q+2n+1)}+2Z_{(q+1,q+2n+1)},
	\end{equation*}
	after getting rid of the factors, we obtain
	\begin{equation}
	Z^\prime_{(p+1,q+2n+1)}=Z_{(p+1,q+2n+1)}+Z_{(q+1,q+2n+1)}.
	\end{equation}
	{However,} this is true from the inductive supposition.
\end{proof}
Note that Formula (\ref{eqq2}) is not true for greater values of $N$. Region II indeed can be made into a parallelogram, similar to region I, since the set of boundary points will remain the same. However, the region corresponding to this parallelogram being translated by $T$ contains new boundary points, where (\ref{eqq2}) does not hold and Corollary \ref{BD_cor_cor} cannot be used further, even though these regions are congruent to each other. The formula for greater values of $N$ needs to include some new terms. In this parallelogram-like region, we need to take into account the reflection of paths, induced by the term $Z_{(M+2l,N)}$ in $Z^\prime_{(M,N)}$, from the filter restriction $\mathcal{F}^1_{l(k+2)-1}$. This is achieved by means of the first part of Lemma \ref{counting_lemma}. Now consider the triangular region, which, similarly to region II being below region I, is below the parallelogram-like region considered previously. There, we need to take into account long steps, acting on paths induced by the term $Z_{(M+2l,N)}$, which have descended to $(l(k+2)-2,N)$ and were acted upon by long steps for the second time. This is being conducted in a similar fashion to Corollary \ref{BD_cor_cor}, where $Z_{(M+2l,N)}$ is assumed to be known from the second part of Lemma \ref{counting_lemma}. This situation for the case of the auxiliary lattice path model in the presence of long steps will be elaborated upon later.

\begin{corollary}\label{lstep_property_cor}
	Fix $j,k\in\mathbb{N}$, $j\leq k$. Let
	\begin{equation*}
	Z_{(M,N)}\equiv Z(L_N((0,0)\to(M,N));\mathbb{S}\;|\;\mathcal{W}^L_0,\{\mathcal{F}^1_{nl-1}\}_{n=j}^{\infty})
	\end{equation*}
	be the weighted number of lattice paths from $(0,0)$ to $(M,N)$ with filter restrictions $\{\mathcal{F}^1_{nl-1}\}_{n=j}^{\infty}$ and set of unrestricted elementary steps $\mathbb{S}$. Let
	\begin{equation*}
	Z^\prime_{(M,N)}\equiv Z(L_N((0,0)\to(M,N));\mathbb{S}\cup\mathbb{S}(k)\;|\;\mathcal{W}^L_0,\{\mathcal{F}^1_{nl-1}\}_{n=j}^{\infty})
	\end{equation*}
	be the weighted number of lattice paths from $(0,0)$ to $(M,N)$ with the same restrictions, with steps $\mathbb{S}\cup\mathbb{S}(k)$.
	Then, for $lk-1\leq M\leq l(k+2)-2$ we have
	\begin{equation}
	Z^\prime_{(M,N)}=Z_{(M,N)},\quad \text{if $N\leq M+ 2l-2$},
	\end{equation}
	\begin{equation}\label{eqq3}
	Z^\prime_{(M,N)}=Z_{(M,N)}+Z_{(M+2l,N)},\quad \text{if  $M+ 2l\leq N \leq l(k+4)-2$}.
	\end{equation}
\end{corollary}
\begin{proof}
	The proof is the same, as for Lemma \ref{lstep_property}. When proving the inductive step, we still can apply Corollary \ref{BD_cor_cor} as region II is still congruent to the one, translated by $T$.
\end{proof}
Note, that Formula (\ref{eqq3}), unlike (\ref{eqq2}), is true for greater values of $N$, as making region II into a parallelogram-like region will not add new boundary points. The manifestation of this fact is that we do not need to take into account the reflection of paths, as they have already been dealt with in term $Z_{(M+2l,N)}$ due to the periodicity of filter restrictions. So, for such a region Formula (\ref{eqq3}) holds. However, for the triangular region below the same problem remains.

Consider the auxiliary lattice path model in the presence of the sequence of steps $\mathbb{S}(k)$.
\begin{definition}\label{multws}
	We denote by multiplicity function in the $j$-th strip $\tilde{M}^j_{(M,N)}$ the weighted number of paths in set 
	\begin{equation*}
	L_N((0,0)\to (M,N);\mathbb{S}\cup\tilde{\mathbb{S}}\;|\;\mathcal{W}^L_0,\{\mathcal{F}_{nl-1}^{1}\}, n\in \mathbb{N})
	\end{equation*}
	with the endpoint $(M,N)$ that lies within $(j-1)l-1 \leq M \leq jl-2$
	\begin{equation}
	\tilde{M}^j_{(M,N)} = Z(L_N((0,0)\to (M,N);\mathbb{S}\cup\tilde{\mathbb{S}}\;|\;\mathcal{W}^L_0,\{\mathcal{F}_{nl-1}^{1}\}, n\in \mathbb{N})),
	\end{equation}
	where $\tilde{\mathbb{S}}$ is a set of some additional steps and $M\geq 0$ and $j=\Big[\frac{M+1}{l}+1\Big]$. 
\end{definition} 
In this subsection, $\tilde{\mathbb{S}}=\mathbb{S}(k)$ if not stated otherwise.
\begin{lemma}\label{l1} For fixed $k\in\mathbb{N}$
	\begin{equation}\label{k1}
	\tilde{M}^{k+1}_{(M,N)}=\sum_{\substack{j=0}}^{[\frac{N-lk+1}{2l}]} M^{k+1+2j}_{(M+2jl,N)},
	\end{equation}
	\begin{equation}\label{k3}
	\tilde{M}^{k+3}_{(M,N)}=\sum_{\substack{j=0}}^{[\frac{N-l(k+2)+1}{2l}]} M^{k+3+2j}_{(M+2jl,N)},
	\end{equation}
	where $\tilde{M}^{j}_{(M,N)}$ is the multiplicity function for $j$-th strip in the auxiliary model with steps $\mathbb{S}\cup\tilde{\mathbb{S}}$, $M^{j}_{(M,N)}$ is the multiplicity function for $j$-th strip in the auxiliary model with steps $\mathbb{S}$.
\end{lemma}
\begin{proof}
	We proceed by induction over $n$, where $n=[\frac{N-l(k+2)+2}{2l}]$, first proving (\ref{k3}), then (\ref{k1}). For $n=0$, Formula (\ref{k3}) follows immediately from the Theorem \ref{BD_cor}, as long steps do not impact this region. Formula (\ref{k1}) follows from Corollary \ref{lstep_property_cor}. As was discussed, Formula (\ref{eqq3}) is true for greater values of $N$, mainly, it is true for a parallelogram-like region, satisfying $n=0$. So, we obtained the base of induction.
	
	Suppose, that
	\begin{equation}\label{ind1}
	\tilde{M}^{k+1}_{(M,N)}=\sum_{\substack{j=0}}^{n+1} M^{k+1+2j}_{(M+2jl,N)}
	\end{equation}
	\begin{equation}\label{ind3}
	\tilde{M}^{k+3}_{(M,N)}=\sum_{\substack{j=0}}^{n} M^{k+3+2j}_{(M+2jl,N)}
	\end{equation}
	is true.
	
	Need to prove the inductive step for (\ref{ind3}) first, thus we need to prove (\ref{k3}) for $M+2ln\leq N\leq M+2l(n+1)$, where $l(k+2)-1\leq M\leq l(k+3)-2$. Denote this region as $\mathcal{L}_1$. The l.h.s. of (\ref{k3}) is a weighted number of paths, $\partial\mathcal{L}_1$ consists of points $(l(k+2)-1,N)$ for $l(k+2)-1+2ln\leq N\leq l(k+2)-1+2l(n+2)$ and $(l(k+2)-1+j,l(k+2)-1+2ln+j)$ for $j=0,\ldots,l-1$. We divide the r.h.s. of (\ref{k3}) into two terms. The first corresponds to the sum given by inductive supposition in (\ref{ind3}). It is also a weighted number of paths defined for $\mathcal{L}_2=\mathcal{L}_1$ with the same boundary points $\partial\mathcal{L}_2=\partial\mathcal{L}_1$, as the l.h.s. of (\ref{k3}). These two regions are congruent, $T_1=id$. The second is an additional term, which we expect to appear during an inductive step. It is given by
	\begin{equation*}
	M^{k+3+2(n+1)}_{(M+2(n+1)l,N)}= Z(L_N((0,0)\to(M+2(n+1)l,N));\mathbb{S}\;|\;\mathcal{W}^L_0,\{\mathcal{F}^1_{ml-1}\}_{m=1}^{\infty})
	\end{equation*}
	for region $l(k+2+2(n+1))-1\leq M\leq l(k+3+2(n+1))-2$ and $M\leq N\leq M+2l$. Denote it by $\mathcal{L}_3$. Its boundary $\partial\mathcal{L}_3$ consists of points $(l(k+2+2(n+1))-1,N)$ for $l(k+2+2(n+1))-1\leq N\leq l(k+2+2(n+1))-1+2l$. This region is an image of $\mathcal{L}_1$ under translation $T_2(M,N)=(M+2l(n+1),N)$, they are congruent. By Corollary \ref{BD_cor_cor}, it is sufficient to prove inductive step at points $(l(k+2)-1,N)$ for $l(k+2)-1+2ln\leq N\leq l(k+2)-1+2l(n+2)$ and points $(l(k+2)-1+j,l(k+2)-1+2ln+j)$ for $j=0,\ldots,l-1$. These are drawn in Figure \ref{proofpic}.
	
	Consider points of a form $(l(k+2)-1,N)$. At $n$-th iteration we added $M^{k+1+2(n+1)}_{(M+2(n+1)l,N)}$ to $\tilde{M}^{k+1}_{(M,N)}$. This term induces paths, which further descend from $(k+1)$-th strip to boundary points of $(k+3)$-th strip. The region in which induced paths descend is congruent to the region, where paths corresponding to $M^{k+1+2(n+1)}_{(M+2(n+1)l,N)}$ continue to descend to the boundary of $(k+3+2(n+1))$-th strip in the auxiliary lattice path model. This is due to the periodicity of filter restrictions. Here, we can apply Theorem \ref{BD_cor} to conclude that the weighted number of induced paths arriving at the boundary of $(k+3)$-th strip is equal to $M^{k+3+2(n+1)}_{(2(n+1)l,N)}$.
	
	Consider points of a form $(l(k+2)-1+j,l(k+2)-1+2ln+j)$. For such points, the proof is the same as for the Formula (\ref{eqq2}) for the hypotenuse of region II.
	
	Now that we proved the inductive step for the boundary of the considered region, by Corollary \ref{BD_cor_cor}, it follows that 
	\begin{equation}
	\tilde{M}^{k+3}_{(M,N)}=\sum_{\substack{j=0}}^{n} M^{k+3+2j}_{(M+2jl,N)}+M^{k+3+2(n+1)}_{(M+2(n+1)l,N)}=\sum_{\substack{j=0}}^{n+1} M^{k+3+2j}_{(M+2jl,N)},
	\end{equation}
	is true for the whole region, which proves the inductive step for Formula (\ref{ind3}).
	\begin{figure}[h!]
		{\includegraphics[width=250pt]{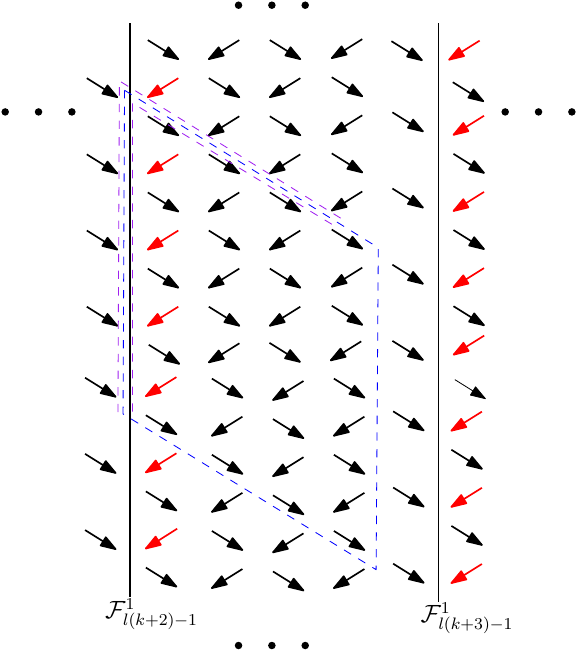}}
		\caption{Region $\mathcal{L}_1$, for which it is sufficient to prove (\ref{k3}) consists of points $M+2ln\leq N\leq M+2l(n+1)$ where $l(k+2)-1\leq M\leq l(k+3)-2$. It is highlighted with blue dashed lines. Union of boundaries for all terms of the considered expression $\partial\mathcal{L}_1\cup T_1^{-1}(\partial\mathcal{L}_2)\cup T_2^{-1}(\partial\mathcal{L}_3)$ consists of points $(l(k+2)-1,N)$ for $l(k+2)-1+2ln\leq N\leq l(k+2)-1+2l(n+2)$ and points $(l(k+2)-1+j,l(k+2)-1+2ln+j)$ for $j=0,\ldots,l-1$. It is highlighted with purple dashed lines. Here, we depict the case, where $l=5$.}
		\label{proofpic}
	\end{figure}
	
	This process for the first iterations is shown in Figure \ref{inducedpath}.

	\begin{figure}[h!]
		{\includegraphics[width=340pt]{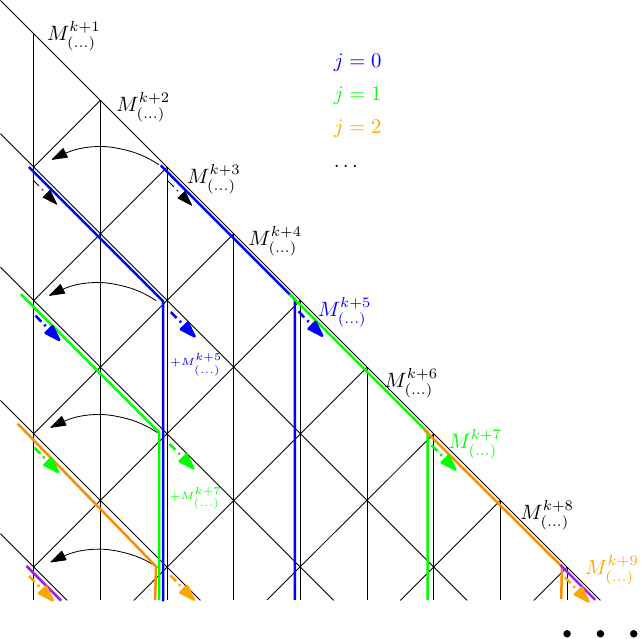}}
		\caption{Color emphasizes the number of iterations in the induction. Paths induced at the boundary of $(k+1)$-th strip during $(j-1)$-th iteration descend in the region, highlighted with color, corresponding to $j$-th iteration. Colored lines outline regions congruent to each other. Dashed colored arrows denote weighted numbers of induced paths, inflicted to $(k+3)$-th strip once they have descended, and their equivalents in strips of the auxiliary lattice path model.}
		\label{inducedpath}
	\end{figure}
	
	This figure also shows how long steps act on descended paths corresponding to dashed colored arrows, inducing paths at boundary points of $(k+1)$-th strip, highlighted with the dashed arrow of the same color. These induced paths, in turn, descend in the region, highlighted with a color corresponding to the next, $(j+1)$-th iteration. Proving that long steps induce paths at boundary points of $(k+1)$-th strip following this scenario amounts to proving the inductive step for Formula (\ref{ind1}).

	Now we need to prove the inductive step for Formula (\ref{ind1}), which amounts to proving (\ref{k1}) for $lk-1\leq M\leq l(k+1)-2$ and $M+2ln\leq N\leq M+2l(n+1)$. It is being conducted in a fashion similar to the proof of (\ref{k3}). Again, we divide the r.h.s. of (\ref{k1}) in two terms. The first one corresponds to the sum given by inductive supposition in (\ref{ind1}). The second one is an additional term, which we expect to appear during an inductive step. It is given by
	\begin{equation*}
	M^{k+1+2(n+1)}_{(M+2(n+1)l,N)}= Z(L_N((0,0)\to(M+2(n+1)l,N));\mathbb{S}\;|\;\mathcal{W}^L_0,\{\mathcal{F}^1_{ml-1}\}_{m=1}^{\infty})
	\end{equation*}
	for region $l(k+2(n+1))-1\leq M\leq l(k+1+2(n+1))-2$ and $M\leq N\leq M+2l$. By Corollary \ref{BD_cor_cor}, it is sufficient to prove inductive step at points $(lk-1,N)$ for $lk-1+2ln\leq N\leq l(k+1)-1+2l(n+2)$ and points $(lk-1+j,lk-1+2ln+j)$ for $j=0,\ldots,l-1$. It is shown in Figure \ref{proofpic1}. This region is the same, as depicted in Figure \ref{proofpic}, but translated by $T(M,N)=(M-2l,N)$.
	
	\begin{figure}[h!]
		{\includegraphics[width=250pt]{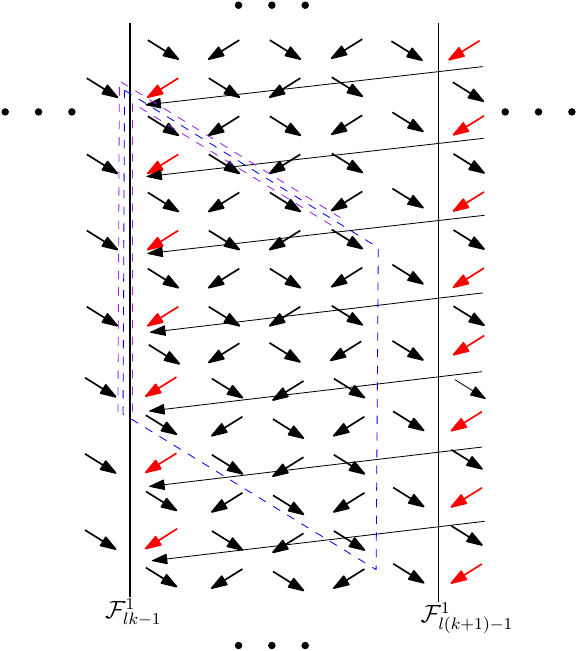}}
		\caption{Points $M+2ln\leq N\leq M+2l(n+1)$ where $l(k+2)-1\leq M\leq l(k+3)-2$ are highlighted with blue dashed lines. Union of boundaries for all terms of the considered expression consists of points $(l(k+2)-1,N)$ for $l(k+2)-1+2ln\leq N\leq l(k+2)-1+2l(n+2)$ and points $(l(k+2)-1+j,l(k+2)-1+2ln+j)$ for $j=0,\ldots,l-1$. They are highlighted with purple dashed lines. Here, we depict the case where $l=5$.}
		\label{proofpic1}
	\end{figure}
	Consider points of a form $(lk-1,N)$. Above, we have seen that Formula (\ref{k3}) receives term $M^{k+3+2(n+1)}_{(M+2(n+1)l,N)}$ during the inductive step. By inductive supposition (\ref{ind1}), it is left to account for the action of long steps, acting on paths, induced by this term. Denote the weighted number of paths, corresponding to this term as
	\begin{eqnarray*}
		Z_{(M,N)}\equiv Z(L_N((-2l(n+1),0)\to(M,N));\mathbb{S}\;|\;\mathcal{W}^L_{-2l(n+1)},\{\mathcal{F}^1_{ml-1-2l(n+1)}\}_{m=1}^{\infty})=\\
		=M^{k+3+2(n+1)}_{(M+2(n+1)l,N)},
	\end{eqnarray*}
	where $lk-1\leq M\leq l(k+2)-1$, $M+2ln\leq N\leq M+2l(n+1)$. Now, we want to calculate
	\begin{equation*}
	Z^\prime_{(M,N)}\equiv Z(L_N((-2l(n+1),0)\to(M,N));\mathbb{S}\cup\tilde{\mathbb{S}}\;|\;\mathcal{W}^L_{-2l(n+1)},\{\mathcal{F}^1_{ml-1-2l(n+1)}\}_{m=1}^{\infty}).
	\end{equation*}
	{From} Corollary \ref{lstep_property_cor}, it is given by
	\begin{equation*}
	Z^\prime_{(M,N)}=Z_{(M,N)}+Z_{(M+2l,N)},
	\end{equation*}
	where
	\begin{equation*}
	Z_{(M+2l,N)}=M^{k+3+2(n+2)}_{(M+2(n+2)l,N)}.
	\end{equation*}
	
	Consider points of a form $(lk-1+j,lk-1+2ln+j)$. For such points, the proof is the same as for Formula (\ref{eqq2}) for the hypotenuse of region II.
	
	Now that we have proven the inductive step for the boundary of the considered region, by Corollary \ref{BD_cor_cor}, it follows that 
	\begin{equation}
	\tilde{M}^{k+1}_{(M,N)}=\sum_{\substack{j=0}}^{n+1} M^{k+1+2j}_{(M+2jl,N)}+M^{k+1+2(n+2)}_{(M+2(n+2)l,N)}=\sum_{\substack{j=0}}^{n+1} M^{k+1+2j}_{(M+2jl,N)},
	\end{equation}
	is true for the whole region, which proves the inductive step for Formula (\ref{ind1}).
\end{proof}
\begin{corollary}\label{l2}
	For fixed $k\in\mathbb{N}$ and $m\geq k$
	\begin{equation}
	\tilde{M}^{m+1}_{(M,N)}=\sum_{\substack{j=0}}^{[\frac{N-lm+1}{2l}]} M^{m+1+2j}_{(M+2jl,N)},
	\end{equation}
\end{corollary}
\begin{proof}
	The result of Lemma \ref{l1} can be extended to other strips in a similar fashion to the proof of (\ref{k3}). Each new term $M^{k+1+2j}_{(M+2jl,N)}$ in $\tilde{M}^{k+1}_{(M,N)}$ induces paths, which further descend from $(k+1)$-th strip to boundary points of each consequent $(k+1+m)$-th strip. The region in which these induced paths descend is congruent to the region, where they would continue to descend in the auxiliary path model due to the periodicity of filter restrictions. Hence, each $\tilde{M}^{k+1+m}_{(M,N)}$ acquires term $M^{k+1+m+2j}_{(M+2jl,N)}$, which proves the statement.
\end{proof}
During this subsection, we introduced long steps and proved lemmas, necessary for counting weighted numbers of paths in modification of the auxiliary lattice path model, relevant for the representation theory of $U_q(sl_2)$ at roots of unity.

\section{On Decomposition of $T(1)^{\otimes N}$ for $U_q(sl_2)$ at Roots of Unity}\label{RelU}
Consider the auxiliary lattice path model with filter restrictions of type $1$, in the presence of steps
\begin{equation*}
\mathbb{S}_U\equiv\mathbb{S}\cup\Big(\bigcup_{k=1}^\infty\mathbb{S}(k) \Big),
\end{equation*}
and denote it as $\mathcal{L}_U$. The arrangement of steps for points of $\mathcal{L}_U$ is depicted in {Figure} \ref{malpm}.

\begin{figure}[h!]
	{\includegraphics[width=280pt]{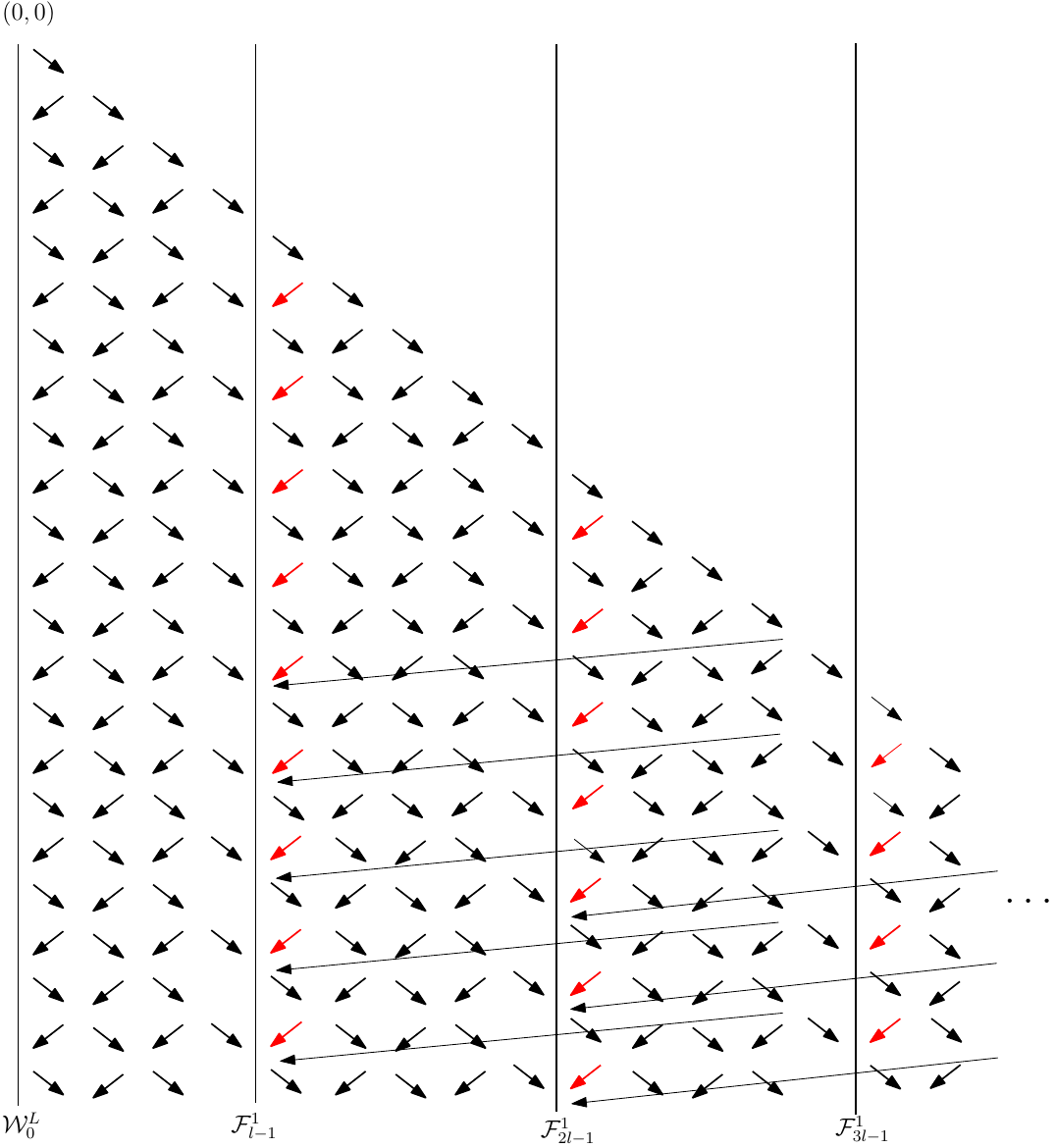}}
	\caption{Arrangement of steps for points of the lattice path model $\mathcal{L}_U$. Here, we depict the case, where $l=5$.}
	\label{malpm}
\end{figure}

Now, fix $q=e^{\frac{\pi i}{l}}$ and $l$ is odd. Category $\mathbf{Rep}(U_q(sl_2))$ is the category of representations of $U_q(sl_2)$, a quantized universal enveloping algebra of $sl_2$ with divided powers. Consider tensor product decomposition of a tensor power of fundamental $U_q(sl_2)$-module
\begin{equation}
T(1)^{\otimes N}=\bigoplus_{k=0}^{N} M^{(l)}_{T(k)}(N)T(k),\quad T(1),\,T(k)\in\mathbf{Rep}(U_q(sl_2)),
\end{equation} 
where $M^{(l)}_{T(k)}(N)$ is the multiplicity of $T(k)$ in tensor product decomposition. We consider the tensor powers of a tilting module, and as a category of tilting modules is closed under finite tensor products, it can be decomposed into a direct sum of tilting modules. The highest weight of $T(k)$ can be written as $k=lk_1+k_0$. The Grothendieck ring of the category of tilting modules over $U_q(sl_2)$ at odd roots of unity gives the following tensor product \mbox{rules (\cite{TPR,LPRS})}
\begin{equation*}
T(k_0) \otimes T(1) = T(k_0+1) \oplus T(k_0 -1),  \quad  0 \leq  k_0 \leq l-2;
\end{equation*}
\begin{eqnarray*}
	T(lk_1 + k_0) \otimes T(1) = T(lk_1 + k_0 +1) \oplus T(lk_1 + k_0 -1) ,\quad 1 \leq k_0 \leq l-3, k_1\geq 1;
\end{eqnarray*}
\begin{eqnarray*} 
	T(lk_1 + l-2) \otimes T(1) = T(l(k_1 +1) -3) \oplus T((k_1+1)l-1) \oplus T((k_1-1)l-1) ,\quad  k_1\geq 1;
\end{eqnarray*}
\begin{equation*}
T(l k_1 -1) \otimes T(1) = T(lk_1) , \quad  k_1 \geq 1;
\end{equation*}
\begin{equation*}
T(lk_1) \otimes T(1) = T(lk_1 +1) \oplus 2 T(lk_1 -1) , \quad  k_1\geq 1.
\end{equation*}

\begin{theorem}[\cite{LPRS}]
	The multiplicity of the tilting $U_q(sl_2)$-module $T(k)$ in the decomposition 
	of $T(1)^{\otimes N}$ is equal to the weighted number of lattice paths on $\mathcal{L}_U$ connecting $(0,0)$ and $(k,N)$ with weights given by multiplicities of elementary steps $\mathbb{S}_U$.
\end{theorem}
\begin{proof}
	Tensor product rules allow the following recursive description of multiplicities
	\begin{equation*}
	M^{(l)}_{T(0)}(N+1)=M^{(l)}_{T(1)}(N);
	\end{equation*}
	\begin{eqnarray*}
		M^{(l)}_{T(lk_1+k_0)}(N+1)=M^{(l)}_{T(lk_1+k_0-1)}(N)+M^{(l)}_{T(lk_1+k_0+1)}(N), \quad 1\leq k_0 \leq l-3,\, k_1\geq 0;
	\end{eqnarray*}
	\begin{equation*}
	M^{(l)}_{T(lk_1-2)}(N+1)=M^{(l)}_{T(lk_1-3)}(N),\quad k_1\geq 1;
	\end{equation*}
	\begin{eqnarray*}
		M^{(l)}_{T(lk_1-1)}(N+1)=M^{(l)}_{T(lk_1-2)}(N)+2M^{(l)}_{T(lk_1)}(N)+M^{(l)}_{T((k_1+2)l-2)}(N),\quad k_1\geq 1;
	\end{eqnarray*}
	\begin{equation*}
	M^{(l)}_{T(lk_1)}(N+1)=M^{(l)}_{T(lk_1-1)}(N)+M^{(l)}_{T(lk_1+1)}(N),\quad k_1\geq 1.
	\end{equation*}
	{This} recursion coincides with the recursion for weighted numbers of paths descending from $(0,0)$ to $(k,N)$ in lattice path model $\mathcal{L}_U$. The latter is depicted in Figure \ref{malpm}.
\end{proof}
The main goal of the following section is to obtain the explicit formula by combinatorial means, mainly counting lattice paths in modification $\mathcal{L}_U$ of the auxiliary lattice path model.

\section{Counting Paths}\label{PATHS}
Consider the lattice path model $\mathcal{L}_U$. From now on, following Definition \ref{multws}, we denote by multiplicity function in the $j$-th strip $\tilde{M}^j_{(M,N)}$ the weighted number of paths in set 
\begin{equation*}
L_N((0,0)\to (M,N);\mathbb{S}_U\;|\;\mathcal{W}^L_0,\{\mathcal{F}_{nl-1}^{1}\}, n\in \mathbb{N})
\end{equation*}
with the endpoint $(M,N)$ that lies within $(j-1)l-1 \leq M \leq jl-2$
\begin{equation}
\tilde{M}^j_{(M,N)} = Z(L_N((0,0)\to (M,N);\mathbb{S}_U\;|\;\mathcal{W}^L_0,\{\mathcal{F}_{nl-1}^{1}\}, n\in \mathbb{N})),
\end{equation}
where 
\begin{equation*}
\mathbb{S}_U=\mathbb{S}\cup\Big(\bigcup_{k=1}^\infty\mathbb{S}(k) \Big).
\end{equation*}
and $M\geq 0$ and $j=\Big[\frac{M+1}{l}+1\Big]$. The main goal of this section is to derive an explicit formula for $\tilde{M}^j_{(M,N)}$.

\begin{lemma}
	For the lattice path model $\mathcal{L}_U$
	\begin{equation}
	\tilde{M}^{1}_{(M,N)}=M^{1}_{(M,N)},
	\end{equation}
	and for $k\in\mathbb{N}$,
	\begin{equation}\label{f1}
	\tilde{M}^{k+1}_{(M,N)}=\sum_{j=0}^{[\frac{N-lk+1}{2l}]} F^{(k-1+2j)}_{k-1} M^{k+1+2j}_{(M+2jl,N)},
	\end{equation}
\end{lemma}
\begin{proof}
	The formula for the $1$st strip follows immediately as long steps have no impact and multiplicity is the same as in the auxiliary lattice path model.
	
	This lemma follows from gradually adding each $\mathbb{S}(k)$ for $k=1,2,\ldots$ to the initial set of steps $\mathbb{S}$ and applying results of the Corollary \ref{l2} repeatedly. Let us start with $k=1$. From Corollary \ref{l2}, it follows that in case of having one series of long steps, for $(m+1)$-th strip we would simply have
	\begin{equation}\label{f2}
	\tilde{M}^{m+1}_{(M,N)}\lvert_{\textit{$k=1$}}=\sum_{\substack{j=0}}^{[\frac{N-lm+1}{2l}]} M^{m+1+2j}_{(M+2jl,N)},
	\end{equation}
	where $m\in\mathbb{N}$. This is a summation of multiplicities in the auxiliary lattice path model with trivial coefficients. This situation is depicted in Figure \ref{U30}.

	\begin{figure}[h!]
		{\includegraphics[width=300pt]{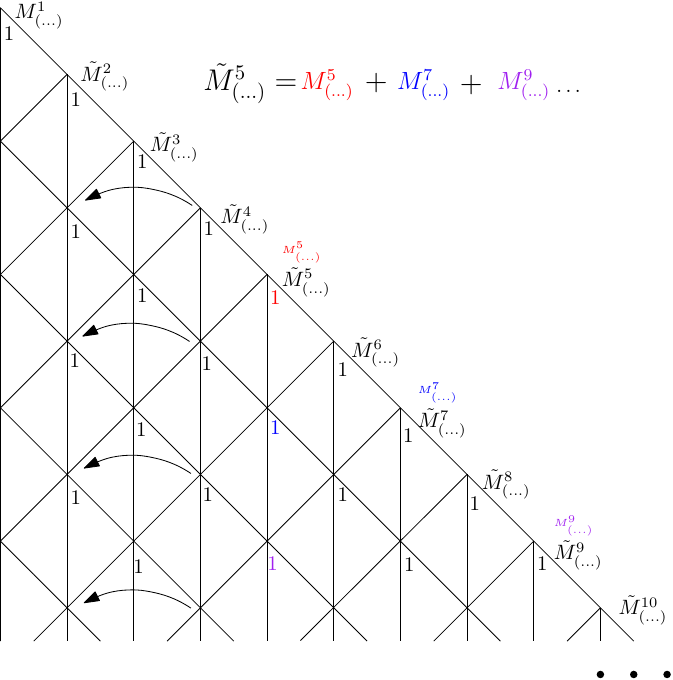}}
		\caption{In each strip, triangles with coefficients correspond to terms in Formula (\ref{f2}). Each term is given by the multiplicity of a strip in the auxiliary lattice path model, situated to the far right of the considered triangle. As an example, we show how this mnemonic rule works for $\tilde{M}^{5}_{(M,N)}$.}
		\label{U30}
	\end{figure}

	In each strip, triangles with coefficients correspond to terms in Formula (\ref{f2}). Each term is given by the multiplicity function of a strip in the auxiliary lattice path model, situated to the far right of the considered triangle. The coefficient in a triangle tells us how many terms corresponding to this multiplicity function are in Formula (\ref{f2}). This mnemonic rule comes from considerations in Figure \ref{inducedpath}. The proofs of Theorem \ref{l1} and Corollary \ref{l2} define a recursion on the coefficients near multiplicity functions from the auxiliary lattice path model in Formula (\ref{f2}). This recursion is depicted in Figure \ref{U31}.

	\begin{figure}[h!]
		{\includegraphics[width=400pt]{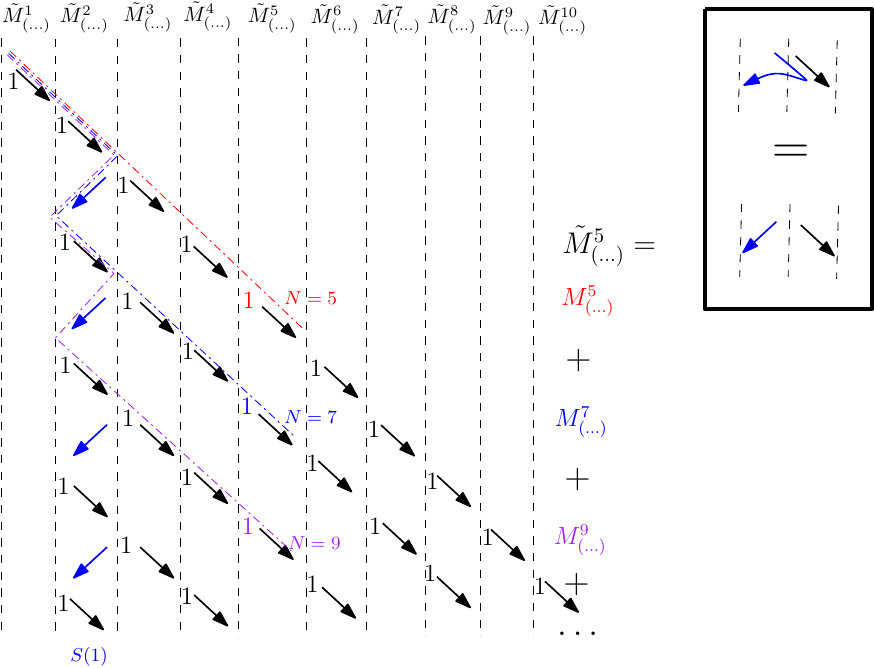}}
		\caption{Numbers near vertices of the lattice are the coefficients in Figure \ref{U30}. Blue arrows denote steps in the recursion, which were added by the long steps $\mathbb{S}(1)$ in the lattice path model. Length of paths $N$, descending to a considered vertex of a lattice gives the number of the strip in the auxiliary lattice path model, the multiplicity function of which is being added to (\ref{f2}) as a term. As an example, we show formula for $\tilde{M}^{5}_{(M,N)}\lvert_{\textit{$k=1$}}$.}
		\label{U31}
	\end{figure}

	Blue arrows correspond to steps in the recursion on the coefficients, which were added as a consequence of the presence of long steps $\mathbb{S}(1)$. In the black frame, it is noted that although long steps have length $2l$, as shown, for example, in Figure \ref{U30}, their source and target points belong to two adjacent strips, so when dealing with the coefficients it is convenient to denote blue arrows as in the Figure \ref{U31}. Long steps, following the idea of the proof of Formula (\ref{k3}) depicted in Figure \ref{inducedpath}, induce paths that descend further, giving the result as in Corollary \ref{l2}. Similarly, blue arrows induce paths in the lattice, which descend further, adding new terms in (\ref{f2}).
	
	Note, that without blue arrows we would have obtained a single diagonal path with weighted numbers of paths equal to $1$. This situation would give us coefficients as in the formula for multiplicities in the auxiliary lattice path model, meaning that we would have $\tilde{M}^k_{(\ldots)}=M^k_{(\ldots)}$. This is exactly what we would have in case we removed the long steps $\mathbb{S}(1)$ in the lattice path model.
	
	Again, following the idea of the proof of Corollary \ref{l2}, induced paths descend further to each consequent strip as if they were to continue to descend in the auxiliary lattice path model, so additional terms are dependent on how many strips these induced paths will cross while they descend. In the recursion on the coefficients, it is manifested in the fact that the length of a descending path in Figure \ref{U31} gives the number of strips in the auxiliary lattice path model, to which the additional term corresponds.

	Now, our main goal is to apply $\mathbb{S}(k)$ for other $k$. As the considerations above suggest, applying $\mathbb{S}(k)$ for $k=1,2,\ldots$ induces other sequences of blue arrows. {From Figure} \ref{U32}, we see that the recursion for the coefficients near multiplicity functions is satisfied by Catalan numbers. This proves Formula (\ref{f1}).
	
	\vspace{-6pt}
	\begin{figure}[h!]
		{\includegraphics[width=350pt]{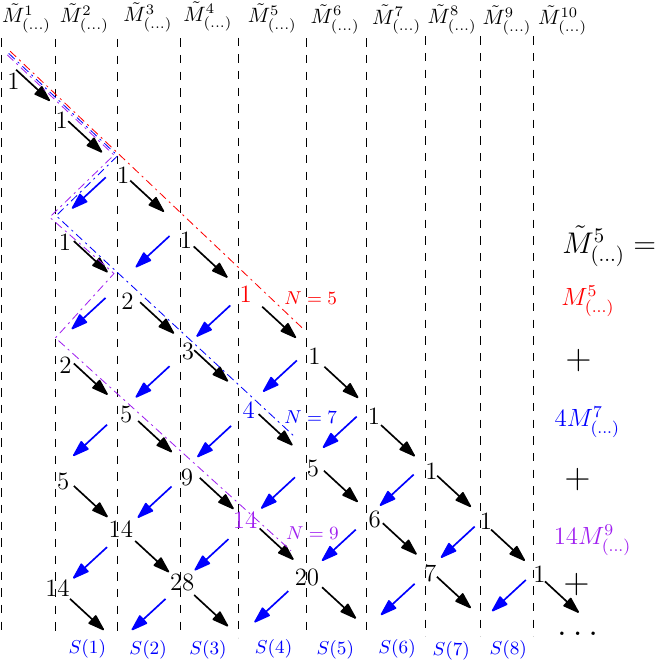}}
		\caption{In each strip, numbers in vertices of a lattice correspond to terms in Formula (\ref{f2}). Each term is given by the multiplicity of a strip in the auxiliary lattice path model, the number of which is given by the length of a path, descending to the considered coefficient. As an example, we show formula for $\tilde{M}^{5}_{(M,N)}$.}
		\label{U32}
	\end{figure}

\end{proof}

Note, that action of black arrows on terms in (\ref{f1}) follows from Lemma \ref{counting_lemma} and the periodicity of filter restrictions. The action of blue arrows on terms in (\ref{f1}) follows from Corollary \ref{BD_cor_cor}. Now let us prove the main theorem.
\begin{theorem}\label{Th1} For $k\in\mathbb{N}\cup\{0\}$ we have
	\begin{equation}\label{qq1}
	\tilde{M}^{k+1}_{(M,N)}=F^{(N)}_{M}+\sum_{j=1}^{[\frac{N-lk+1}{2l}+\frac{1}{2}]}F^{(N)}_{-2lk+M-2jl}+\sum_{j=1}^{[\frac{N-lk+1}{2l}]}F^{(N)}_{M+2jl}
	\end{equation}
	where $lk-1\leq M\leq l(k+1)-2$.
\end{theorem}
\begin{proof}
	We proceed by induction over $[\frac{N-lk+1}{2l}+\frac{1}{2}]$. For $[\frac{N-lk+1}{2l}+\frac{1}{2}]=1$ the \mbox{Formula (\ref{qq1})} obviously gives the same result as (\ref{f1}), which is the base of induction.

	Suppose, that
	\begin{equation}\label{iind2}
	\tilde{M}^{k+1}_{(M,N)}=F^{(N)}_{M}+\sum_{j=1}^{n}F^{(N)}_{-2lk+M-2jl}+\sum_{j=1}^{n-1}F^{(N)}_{M+2jl}
	\end{equation}
	is true. We need to prove this statement for $n+1$. It is sufficient to compare coefficients in (\ref{qq1}) and (\ref{f1}) near $F^{(N)}_{M+2nl}$ and $F^{(N)}_{-2lk+M-2(n+1)l}$. We focus on the term $F^{(N)}_{M+2nl}$, the rest can be performed in a similar fashion. From the structure of $M^{k}_{(M,N)}$, given by Theorem \ref{mainps} and depicted in Figure \ref{U5}, we have two cases: $n+1$ is odd and $n+1$ is even.
	\begin{figure}[h!]
		{\includegraphics[width=150pt]{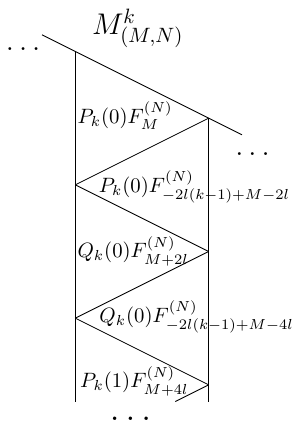}}
		\caption{{Graphical presentation of terms in the formula for $M^{k}_{(M,N)}$, given by Theorem \ref{mainps}. Each term is depicted in accordance to the domain of the lattice where it appears for the first time.}}  
		\label{U5}
	\end{figure}
	For the case of odd $n+1$ the last term in (\ref{f1}) is associated with $P_j(\frac{n}{2})$, for the case of even $n$ it is $Q_j(\frac{n+1}{2})$. We focus on the case of odd $n+1$, the other case can be proven in a similar manner. So, the proof boils down to a combinatorial identity
	\begin{equation}\label{iind1}
	\sum_{\substack{j=0 \\ even\,j}}^{n}F^{(k-1+2j)}_{k-1}P_{k+1+2j}\Big(\frac{n-j}{2}\Big)-\sum_{\substack{j=1 \\ odd\,j}}^{n-1}F^{(k-1+2j)}_{k-1}Q_{k+1+2j}\Big(\frac{n-1-j}{2}\Big)=1.
	\end{equation}
	{From} comparing coefficients near $F^{(N)}_{M+2(n-2)l}$ in the inductive supposition (\ref{iind2}), we know that
	\begin{equation}\label{iind4}
	\sum_{\substack{j=0 \\ even\,j}}^{n-2}F^{(k-1+2j)}_{k-1}P_{k+1+2j}\Big(\frac{n-2-j}{2}\Big)-\sum_{\substack{j=1 \\ odd\,j}}^{n-3}F^{(k-1+2j)}_{k-1}Q_{k+1+2j}\Big(\frac{n-3-j}{2}\Big)=1
	\end{equation}
	is true. Take into account, that
	\begin{equation}
	P_{j}(k)-P_{j}(k-1)=\binom{j+2k-3}{j-3},
	\end{equation}
	\begin{equation}
	Q_{j}(k)-Q_{j}(k-1)=\binom{j+2k-2}{j-3}.
	\end{equation}
	{Subtracting} (\ref{iind4}) from (\ref{iind1}), we obtain
	\begin{eqnarray}
	\Big(\sum_{\substack{j=0 \\ even\,j}}^{n-2}-\sum_{\substack{j=1 \\ odd\,j}}^{n-3}\Big)F^{(k-1+2j)}_{k-1}\binom{j+n+k-2}{2j+k-2}+\\ \nonumber
	+F^{(k-1+2n)}_{k-1}P_{k+1+2n}(0)-F^{(k-1+2n-2)}_{k-1}Q_{k+1+2(n-1)}(0)=0
	\end{eqnarray} 
	{Taking} into account that $P_j(0)=1$ and $Q_j(0)=j-2$ and simplifying further, we arrive at
	\begin{equation}\label{q}
	\sum_{j=0}^{n}(-1)^{j}F^{(k-1+2j)}_{k-1}\binom{j+n+k-2}{2j+k-2}=0.
	\end{equation}
	\begin{equation}
	\sum_{j=0}^{n}(-1)^{j}\binom{j+n+k-2}{2j+k-2}\Big(\binom{2j+k-1}{j}-\binom{2j+k-1}{j-1}\Big)=0
	\end{equation}
	{The} last identity follows from the following lemma.
	\begin{lemma} For $n,k\in\mathbb{N}$
		\begin{equation}\label{onee}
		\sum_{j=0}^{n}(-1)^{j}\binom{j+n+k-2}{2j+k-2}\binom{2j+k-1}{j}=2(-1)^n,
		\end{equation}
		\begin{equation}\label{twoo}
		\sum_{j=0}^{n}(-1)^{j}\binom{j+n+k-2}{2j+k-2}\binom{2j+k-1}{j-1}=2(-1)^n.
		\end{equation}
	\end{lemma}
	\begin{proof}
		Let us first prove Formula (\ref{onee}). Denote
		\begin{equation*}
		F(n,j)=\frac{(-1)^{j+n}}{2}\binom{j+n+k-2}{2j+k-2}\binom{2j+k-1}{j}.
		\end{equation*}
		{We} need to show, that
		\begin{equation*}
		\sum_{j=0}^nF(n,j)=1,\quad\forall n\in\mathbb{N}.
		\end{equation*}
		{For} $n=1$ it is true, which gives us the base of induction. Using Zeilberger's algorithm(\cite{WZ90,Z91,PS94}), we obtain its Wilf--Zeilberger pair
		\begin{eqnarray*}
			G(n,j)=\frac{(-1)^{j+n}j(j+k-1)(k+2n)}{2n(n+1)(j-n-1)(k+2j-1)(k^2+n(n-1)+k(2n-1))}\times\\
			\times(1+k^2n-3n^2+k(n^2-3n-1)+j(2n^2+2kn+k-1))\binom{j+n+k-2}{2j+k-2}\binom{2j+k-1}{j},
		\end{eqnarray*}
		for which
		\begin{equation}
		-F(n+1,j)+F(n,j)=G(n,j+1)-G(n,j)
		\end{equation}
		is true. Applying sum over $j$ to both sides and simplifying telescopic sum to the right, we obtain that
		\begin{equation}
		\sum_{j=0}^nF(n+1,j)=\sum_{j=0}^nF(n,j)+G(n,0)-G(n,n+1).
		\end{equation}
		{Taking} into account, that $G(n,0)=0$ and $G(n,n+1)=F(n+1,n+1)$, we have
		\begin{equation}
		\sum_{j=0}^{n+1}F(n+1,j)=\sum_{j=0}^nF(n,j),
		\end{equation}
		which, considering inductive supposition, proves Formula (\ref{onee}).
		
		Formula (\ref{twoo}) can be proven in a similar fashion, the corresponding Wilf--Zilberger pair is given by
		\begin{eqnarray*}
			G(n,j)=\frac{(-1)^{j+n}(j-1)(j+k)(k+2n)}{2n(n-1)(j-n-1)(k+2j-1)(k^2+n(n-1)+k(2n-1))}\times\\
			\times(1+k^2(n-1)+k(n^2-3n)-3n^2+j(2n^2+2kn-k-1))\binom{j+n+k-2}{2j+k-2}\binom{2j+k-1}{j-1}.
		\end{eqnarray*}
	\end{proof}
	The proof of this lemma concludes the proof of the identity and finishes the proof of the initial theorem. 
\end{proof}

\section{Conclusions}\label{CONCLUSIONS}
In this paper, we considered the lattice path model $\mathcal{L}_U$, which is the auxiliary lattice path model in the presence of long steps. Weighted numbers of paths in this model recreate multiplicities of $U_q(sl_2)$-modules in tensor product decomposition of $T(1)^{\otimes N}$, where $U_q(sl_2)$ is a quantum deformation of the universal enveloping algebra of $sl_2$ with divided powers and $q$ is a root of unity. Explicit formulas for multiplicities of all tilting modules in tensor product decomposition were derived by purely combinatorial means in the main theorem of this paper Theorem \ref{Th1}.

We found that the auxiliary lattice model defined in \cite{PS} is of great use for counting multiplicities of modules of differently defined quantum deformations of $U(sl_2)$ at $q$ root of unity. For instance, in \cite{LPRS} we applied periodicity conditions to the auxiliary lattice path model to obtain a folded Bratteli diagram, weighted numbers of paths for which recreate multiplicities of modules in tensor product decomposition of $T(1)^{\otimes N}$, where $T(1)$ is a fundamental module of the small quantum group $u_q(sl_2)$. In this paper, we modified the auxiliary lattice path model by applying long steps to obtain multiplicities for the case of $U_q(sl_2)$ with divided powers in a similar fashion.\par
The model defined in \cite{PS} required analysis of combinatorial properties of filters, which we heavily relied on. In this paper, we introduced long steps and explored their combinatorial properties. In order to derive formulas for weighted numbers of paths in this setting, we also defined boundary points and congruence of regions in lattice path models. The philosophy of congruence is fairly easy to understand. Two different lattice path models can be locally indistinguishable due to coinciding recursions for weighted numbers of paths in these regions. Weighted numbers of paths at boundary points of the considered region uniquely define weighted numbers of paths for the rest of the region by recursion. So, instead of proving identities for the whole region, it is sufficient to prove such only for boundary points of the region. At boundary points, an identity can be represented as a linear combination of weighted numbers of paths from different lattice path models and one needs to take into account boundary points of congruent regions with respect to all these models.

We found that besides applying periodicity conditions to the auxiliary lattice path model, one can take $U_q(sl_2)$, consider its restriction to $u_q^-U_q^0u_q^+$, where $u_q^\pm$ are subalgebras of the small quantum group $u_q(sl_2)$, generated by $F$ and $E$, respectively, and $U^0_q$ is a subalgebra of $U_q(sl_2)$, generated by $K^{\pm 1}$ and $\qbin{K ; c}{t}$, for $t\geq 0$, $c\in\mathbb{Z}$. Then, we can restrict $u_q^-U_q^0u_q^+$ to $u_q(sl_2)$. This procedure defines another modification of the auxiliary lattice path model and, remarkably, gives the same result as with periodicity conditions. The lattice path model corresponding to $u_q^-U_q^0u_q^+$ will be considered in the upcoming paper.

Considering other possible directions for further research, the following questions remain open:
\begin{itemize}
	\item Multiplicity formulas for decomposition of tensor powers of fundamental representations of $U_q(sl_n)$ at roots of unity remain out of reach and can be a source of inspiration for other interesting combinatorial constructions. We expect that for $U_q(sl_n)$ derivation of such formulas will rely on similar combinatorial ideas. It is worth mentioning that obtaining such formulas explicitly is of interest for asymptotic representation theory, mainly, for constructing Plancherel measure and possibly obtaining its limit shape in different regimes, including regime when $n\to\infty$ (\cite{BOO,PR,LPRS}).
	\item In \cite{STWZ}, similar lattice path models emerge when studying the Grothendieck ring of the category of tilting modules for $U_q(sl_2)$ in the mixed case: when $q$ is an odd root of unity and the ground field is $\overline{\mathbb{F}_p}$. One can expand the combinatorial analysis presented in this paper to a mixed case.
\end{itemize}


\begin{thebibliography}{6}  
	\bibitem[1]{L95}  Littelmann, P.{ Paths and root operators in representation theory}. \emph{Ann. Math.} \textbf{1995}, \emph{142},  499–525.
	\bibitem[2]{M98}  Macdonald, I. {\it Symmetric Functions and Hall Polynomials}, 2nd ed.; Oxford mathematical Monographs; Clarendon Press: {Oxford, UK}, 1998.
	\bibitem[3]{K90}  Kashivara, M. { Crystalizing the $q$-analogue of universal enveloping algebras}. \emph{Commun. Math. Phys.} \textbf{1990}, \emph{133}, 249–260.
	\bibitem[4]{GM93}  Grabiner, D.;   Magyar, P. { Random walks in Weyl chambers and the decomposition of tensor powers}.  \emph{J. Algebr. Combin.} \textbf{1993}, \emph{2}, 239–260.
	\bibitem[5]{G02}  Grabiner, D. { Random walk in an alcove of an affine Weyl group, and non-colliding random walks on an interval}. \emph{J. Comb. Theory Ser. A} \textbf{2002}, \emph{97}, 285–306.
	\bibitem[6]{TZ04}  Tate, T.;  Zelditch, S. { Lattice path combinatorics and asymptotics of multiplicities of weights in tensor powers}. \emph{J. Funct. Anal.} \textbf{2004}, \emph{217}, 402–447.
	\bibitem[7]{PR}  Postnova, O.;   Reshetikhin, N. { On multiplicities of ireducibles in large tensor product of representations of simple Lie algebras}. \emph{arXiv} \textbf{2020}, arXiv:1812.11236.
	\bibitem[8]{B72}  Bratteli, O. { Inductive limits of finite dimensional C*-algebras}. \emph{Trans. Am. Math. Soc.} \textbf{1972}, \emph{171}, 195--234.
	\bibitem[9]{LPRS}  Lachowska, A.;  Postnova, O;  Reshetikhin, N.;   Solovyev, D. { Tensor Powers of Vector Representation of $U_q(\mathfrak{sl}_2)$ at Even Roots of Unity}, TBA. 
	\bibitem[10]{QG}  Faddeev, L.;  Reshetikhin, N.;  Takhtajan, L. { Quantization of Lie groups and Lie algebras}. \emph{Algebr. Anal.} \textbf{1988}, {\emph{1:1}}, {129}--139. 
	\bibitem[11]{CP}  Chari, V.;    Pressley, A. {\it A Guide to Quantum Groups}; Cambridge University Press: {Cambridge, UK,} 1995.
	\bibitem[12]{A92}  Andersen, H. { Tensor products of quantized tilting modules}. \emph{Comm. Math. Phys.} \textbf{1992}, \emph{149}, 149--159.
	\bibitem[13]{PS} Postnova, O.; Solovyev, D. {Counting filter restricted paths in $\mathbb{Z}^2$ lattice}. \emph{arXiv} \textbf{2021}, arXiv:2107.09774.
	\bibitem[14]{SQG}  Lusztig, G. { Finite-dimensional Hopf algebras arising from quantized universal enveloping algebra}. \emph{J. Am. Math. Soc.} \textbf{1990}, \emph{3}, 257–296.
	\bibitem[15]{S}  {Solovyev, D.}{ Towards counting paths in lattice path models with filter restrictions and long steps}. \emph{Zap. Nauch, Sem. POMI} \textbf{2021}, \emph{509}, 201--215.
	\bibitem[16]{Kratt}  Krattenthaler, C. Lattice path combinatorics chapter.  In {\it Handbook of Enumerative Combinatorics};  Bona, M., Ed.;  Chapman and Hall:  {London, UK,} 2015.
	\bibitem[17]{TPR} Andersen, H.H.;    Paradowski, J. { Fusion categories arising from semisimpleLie algebras}. \emph{Comm. Math. Phys.} \textbf{1995}, \emph{169}, 563--588.
	\bibitem[18]{Z91}  Zeilberger,  D. { The Method of Creative Telescoping}.  \emph{J. Symb. Comput.} \textbf{1991}, \emph{11}, 195--204.
	\bibitem[19]{WZ90}  Wilf, H.;    Zeilberger, D. { Rational Functions Certify Combinatorial Identities}. \emph{J. Am. Math. Soc.} \textbf{1990}, \emph{3}, 147--158.
	\bibitem[20]{PS94}  Paule, P.;    Schorn, M. { A Mathematica Version of Zeilberger's Algorithm for Proving Binomial Coefficient Identities}.  \emph{J. Symb. Comput.} \textbf{{1994}}, \emph{{11}}, {673--698}.  
	\bibitem[21]{BOO} Borodin, A.; Okounkov, A.; Olshanski, G. { Asymptotics of Plancherel measures for symmetric groups}.  \emph{J. Am. Math. Soc.} \textbf{2000},\emph{13}, 481--515.
	\bibitem[22]{STWZ}Sutton, L.; Tubbenhauer, D.; Wedrich, P.; Zhu, J. { SL2 tilting modules in the mixed case}. \emph{arXiv} \textbf{2021}, arXiv:2105.07724.
	\bibitem[23]{RT}  Reshetikhin, N.;    Turaev, V. { Invariants of 3-manifolds via link polynomials and quantum groups}. \emph{Invent. Math.} \textbf{1991}, \emph{103}, 547--597.
	\bibitem[24]{M92} Martin, P.P. { On Schur-Weyl duality, $A_n$ Hecke algebras and quantum sl(N) on $\otimes^{n+1}\mathbb{C}^N$}. \emph{Int. J. Mod. Phys. A} \textbf{1992}, \emph{7}, 645--673.
	
\end{thebibliography}
\end{document}